\documentclass[a4paper, 11pt]{article} %fleqn

\usepackage[T1]{fontenc}

\usepackage{lmodern}

\usepackage{amssymb, mathtools}
\usepackage{graphicx}

\usepackage{pgfplots}
\pgfplotsset{compat=1.11}
\usepackage{mathrsfs}
\usepackage{mathtools}
\usepackage{dsfont}
\usepackage{amsthm}
\usepackage{enumerate}
\usepackage{color}
\usepackage{verbatim}
\usepackage[margin=2.3cm, includefoot, footskip=30pt]{geometry}
\usepackage{array}
\newcolumntype{e}{>{\displaystyle}r @{\,} >{\displaystyle}c @{\,} >{\displaystyle}l}

\usepackage{amssymb, amsopn, amscd}

\usepackage{microtype}

\usepackage{babel}

\usepackage{xcolor}
\usepackage[colorlinks]{hyperref}
\hypersetup{final}
\hypersetup{
    colorlinks,
    linkcolor={blue!40!black},
    citecolor={blue!40!black},
    urlcolor={blue!80!black}
}
\usepackage{nameref}
\usepackage[all]{hypcap}

\theoremstyle{plain}
\newtheorem{Theorem}{Theorem}[section]

\newtheorem{Lemma}[Theorem]{Lemma}
\newtheorem{Proposition}[Theorem]{Proposition}

\theoremstyle{definition}

\newtheorem{Remark}[Theorem]{Remark}

\newcommand{\E}{\mathbb E}
\newcommand{\R}{\mathbb R}
\newcommand{\N}{\mathcal N}
\renewcommand{\S}{\mathcal S}
\newcommand{\G}{\mathcal G}
\newcommand{\dev}{\mathbf{dev}}
\newcommand{\F}{\mathcal F}
\newcommand{\1}{\mathds 1}
\renewcommand{\P}{\mathbb P}

\numberwithin{equation}{section}

\begin{document}

\title{The Radial Spanning Tree is straight in  all dimensions}
\author{Tom Garcia-Sanchez\footnote{\href{mailto:tom.garcia-sanchez@imt-nord-europe.fr}{tom.garcia-sanchez@imt-nord-europe.fr}; IMT Nord Europe, France}}
\maketitle

\begin{abstract}
    The \emph{Radial Spanning Tree} (RST) in dimension $d\geq2$ is a \emph{random geometric graph} constructed on a \emph{homogeneous Poisson point process} $\N$ in $\R^d$ augmented by the origin, with edges connecting each $x\in\N$ to the nearest point $y\in\N\cup\{0\}$ that lies closer to $0$ than $x$, with respect to the \emph{Euclidean distance}. By construction, it forms almost surely a \emph{tree} rooted at $0$. The RST was introduced in 2007 by Baccelli and Bordenave, who investigated \emph{straightness}, a deterministic property introduced by Howard and Newman in 2001, to derive information about the \emph{asymptotic directions} of \emph{infinite branches}. They proved that the RST is almost surely \emph{straight} in dimension $2$, which directly implies that all infinite branches are asymptotically directed, every possibility is attained, and directions reached by multiple infinite branches form a dense subset. However, their approach relies crucially on \emph{planarity}, preventing any straightforward extension to higher dimensions. In this paper, we close this gap by proving that the RST is almost surely \emph{straight} in any dimension, thereby obtaining the same consequences for the behavior of infinite branches. Our approach resolves the key barriers in the study of the RST, notably those posed by the complex dependency structure combined with the radial nature of the model, and especially beyond the planar setting. It relies on tools developed for the analysis of the \emph{Directed Spanning Forest}, a closely related model, including recent progress by the author in 2025. Specifically, a key contribution of this work is the construction of a suitable renewal-type decomposition of RST paths. Leveraging this decomposition together with classical concentration inequalities, we show that RST paths cannot deviate far from straight lines and derive straightness.
\end{abstract}

\paragraph{Acknowledgments.}
The author warmly thanks his PhD advisors, David Coupier and Viet Chi Tran, for their continuous guidance and support throughout this work. This research was partly supported by the ANR project \emph{GrHyDy} (ANR-20-CE40-0002), the CEFIPRA project \emph{Directed random networks and their scaling limits} (No.\ 6901), the CNRS RT \emph{MAIAGES} (Action 2179), and by a doctoral fellowship from ENS Paris.

\section{Introduction}

In this paper, we study the \emph{Radial Spanning Tree} (RST) in $\R^d$, for arbitrary fixed dimension $d\geq 2$. The construction proceeds as follows. Consider an \emph{homogeneous Poisson point process} $\N$ of unit intensity in $\R^d$. Let $\|\cdot\|$ denote the standard Euclidean norm of $\R^d$. If $x\in\R^d\setminus\{0\}$ and $X\subset\R^d$ are such that the infimum
    \[\inf_{\substack{y\in X\cup\{0\}\\\|y\|<\|x\|}}\|y-x\|\]
is attained at a unique point, we denote such a minimizer by $\Psi(x, X)$. We set $\Psi(0,X)\coloneqq 0$ by convention. Since with probability one $\N$ is locally finite and contains no nontrivial isosceles triangles, $\Psi(x)\coloneqq\Psi(x,\N)$ is almost surely well defined for every $x\in\N$. The RST is then defined as the \emph{random geometric graph} whose \emph{vertex set} is $\N\cup\{0\}$ and whose \emph{directed edges} are given by
    \[\{(x,\Psi(x)):x\in\N\}.\]
In words, each vertex $x\in\N$ has a unique outgoing edge pointing to the nearest point in $\N\cup\{0\}$ that lies closer to the origin. As $\N$ is locally finite with probability one, the resulting graph is almost surely a \emph{tree} rooted at $0$. Some simulations are presented in Figure~\ref{fig_rst}.

\begin{figure}[h]
    \label{fig_rst}
    \centering
    \includegraphics[height=0.33\linewidth]{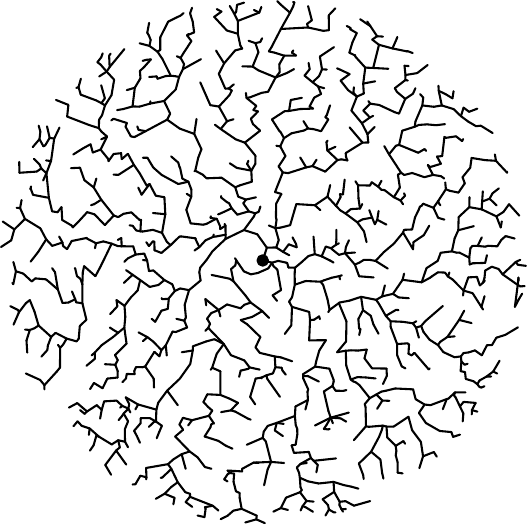}\qquad
    \includegraphics[height=0.33\linewidth]{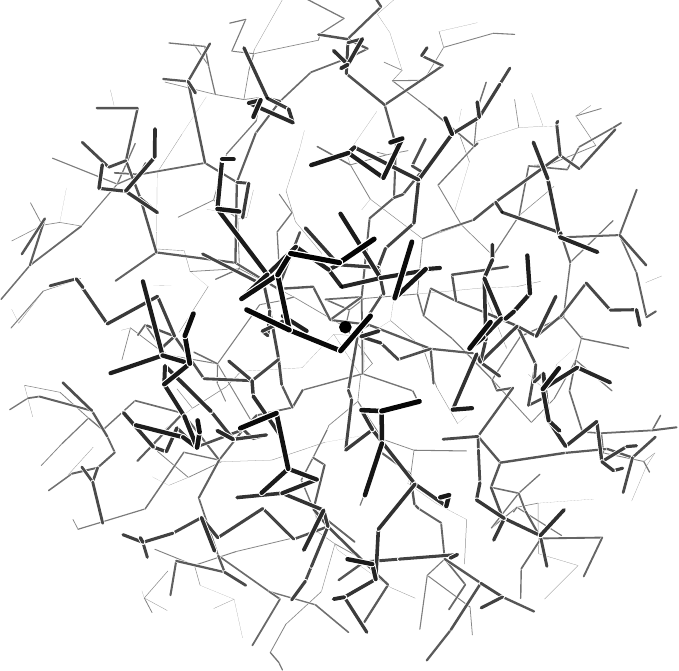}
    \caption{Simulations of the RST within a finite ball around the origin, depicted as a dot. The picture on the left corresponds to $d=2$, while the picture on the right corresponds to $d=3$, with line color and thickness chosen to help visualizing the depth.}
    \label{fig:placeholder}
\end{figure}

Let us now outline the fundamental questions that arise in the study of infinite geometric trees. If $T$ is a fixed \emph{deterministic} tree in $\R^d$ rooted at $0$, an \emph{infinite branch} is defined as a \emph{backward path} $(u_n)_{n\geq0}$ starting at the root, that is a sequence of vertices such that $(u_{n+1}, u_n)$ is an edge of $T$ for all $n\geq0$, and it is said to be \emph{asymptotically directed} toward $\xi\in\mathbb S^{d-1}\coloneqq\{\zeta\in\R^d:\|\zeta\|=1\}$ if
    \[\lim_{n\to\infty}\frac{u_n}{\|u_n\|}=\xi.\]
This leads to several natural structural problems: \emph{How many infinite branches are there? Do they admit asymptotic directions? Which limits are possible? How many distinct infinite branches can share the same one?} A crucial tool to address these questions is the notion of \emph{straightness}, a property introduced by Howard and Newman \cite{howard2001geodesics} in 2001, which we now define precisely. For any vertex $u$ of $T$, denote by $V_u$ the set of vertices in the sub-tree rooted at $u$. For any $\xi\in\mathbb S^{d-1}$ and $\theta\in\R$, we define the \emph{cone} pointing in direction $\xi$ with \emph{aperture angle} $\theta$ as
    \[C[\xi, \theta]\coloneqq\{x\in\R^d:x\cdot\xi\geq\|x\|\cos\theta\},\]
where $\cdot$ denotes the standard scalar product of $\R^d$. The tree $T$ is said to be \emph{straight} if there exists a function $f:\R_+\to\R_+$ with $\lim_{t\to\infty}f(t)=0$ such that for all but finitely many vertices $u\neq 0$ of $T$,
    \[V_u\subset C\left[\frac{u}{\|u\|},f(\|u\|)\right].\]
In words, $T$ is straight if the \emph{angular spread} of its sub-trees vanishes as their root tends to infinity. The relevance of straightness stems from the following key proposition due to \cite{howard2001geodesics}.

\begin{Proposition}[Proposition 2.8 in \cite{howard2001geodesics}]
    \label{prop_straightness_consequences}
    Assume that $T$ is straight, each vertex has a finite degree, and the vertex set $V_0$ of $T$ is asymptotically omnidirectional, i.e.\
        \[\left\{\frac{u}{\|u\|}:u\in V_0,~\|u\|\geq n\right\}\]
    is dense in $\mathbb S^{d-1}$ for each $n\geq 1$. Then,
    \begin{itemize}
        \item[\textnormal{(i)}] Each infinite branch of $T$ admits an asymptotic direction.
        \item[\textnormal{(ii)}] For each $\xi\in\mathbb S^{d-1}$, there exists at least one infinite branch with asymptotic direction $\xi$.
        \item[\textnormal{(iii)}] The set of $\xi\in\mathbb S^{d-1}$ such that there exists more than one infinite branch with asymptotic direction $\xi$ is dense in $\mathbb S^{d-1}$.
    \end{itemize}
\end{Proposition}

In summary, under mild assumptions, establishing straightness provides extensive information about the infinite branches of a tree, and since such a property is deterministic, the ideal goal in the context of random geometric trees is to show that it holds almost surely. This is precisely what we achieve for the RST, stated as the main theorem of this work below. An illustration is provided in Figure~\ref{fig_straight}

\begin{Theorem}
    \label{thm_main}
    For every dimension $d\geq2$, the RST is almost surely straight, and Proposition~\ref{prop_straightness_consequences} applies on an event of full probability.
\end{Theorem}

\begin{figure}[h]
    \centering
    \includegraphics[width=0.5\linewidth]{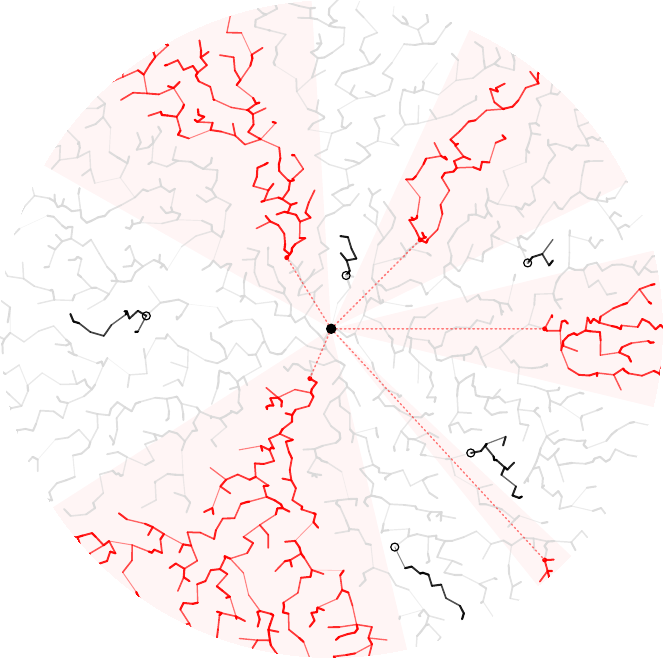}
    \caption{Simulation of the planar RST in a ball centered at the origin. Edges appear in light gray, with thickness decreasing with length. In red are highlighted some sub-trees exiting the displayed area, contained within translucent red cones pointing toward their sub-roots, with arbitrarily chosen apertures that appear to decrease to $0$ with the distance to the origin, suggesting straightness. In black are highlighted some finite sub-trees remaining within the area, illustrating that RST sub-trees are not necessarily infinite.}
    \label{fig_straight}
\end{figure}

We now discuss the origin of the model as well as relevant existing results and related works. The RST was introduced in 2007 by Baccelli and Bordenave \cite{BacceliBordenave}, who proved that it is almost surely straight in dimension~$2$ and, via Proposition~\ref{prop_straightness_consequences}, derived detailed information about its infinite branches.

These results were refined in 2013 by Baccelli, Coupier, and Tran \cite{baccelli2013semi}, who exploited \emph{planarity}, i.e.\ the fact that RST paths can not cross in dimension $2$, as illustrated in Figure~\ref{fig_rst}), to relate so-called \emph{exceptional directions}, those attained by more than one infinite branch, to \emph{competition interfaces}. This approach made it possible to establish the measurability of the set of exceptional directions and to prove that it is almost surely countable. By further combining this with \emph{isotropy} and an application of Fubini’s theorem, they showed that for each fixed deterministic $\xi\in\mathbb S^1$ there is almost surely a unique infinite branch asymptotically directed toward $\xi$. %Additionally, they showed that infinite branches are rare in the sense that the average number of infinite branches crossing the boundary of $B(0,r)$ is asymptotically negligible with respect to $r$.

The infinite branches in the \emph{hyperbolic} analogue of the RST have also been studied in \cite{flammant2020hyperbolic, coupier2023thicktraceinfinityhyperbolic}. In dimension~$2$, the same conclusions as in the Euclidean case hold, with the remarkable strengthening that, almost surely, no direction is attained by more than two infinite branches. In fact, this last property, often referred to as N3G (see, e.g., \cite{n3g_hyperbolic}) for \emph{no three geodesics} in the same direction, has attracted a lot of attention in the literature. Although expected to hold for a large collection of models, the only other examples where it has been proven, to the best of our knowledge, are geodesic trees arising in Exponential Last Passage Percolation \cite{n3g_lpp} and the Directed Landscape \cite{n3g_landscape}. 
In higher dimensions, the equivalent of straightness, coming with the consequences of Proposition~\ref{prop_straightness_consequences}, has been established almost surely. Moreover, it was shown that any fixed deterministic direction is almost surely reached by a unique infinite branch. All of these results crucially exploit the singular features of hyperbolic geometry, allowing an analysis beyond the planar setting.

For completeness’ sake, we also mention that results on local functionals, such as vertex degrees and edge lengths, have been established for the RST in both the Euclidean case and hyperbolic variant. We refer the interested reader to \cite{schulte2017central, rosen2025radial}.

Beyond the RST, straightness has also been investigated in other models. In the seminal work of \cite{howard2001geodesics}, the property was established for the geodesic trees of the \emph{Euclidean First Passage Percolation} in arbitrary dimension. The only other results concerning non-planar dimensions in the Euclidean setting, to our knowledge, are due to \cite{bordenave2008navigation}, who provided strong sufficient conditions for trees constructed on homogeneous Poisson point processes to be straight, and addressed the case of the \emph{small-world maximal progress navigation}. It is important to note, however, that these assumptions are not satisfied by the RST, regardless of the dimension.

In dimension $2$, further instances where straightness has been established include the Poisson Radial Tree \cite{RPTstraightness} as well as the geodesic trees arising in the \emph{Hammersley model} on $\R^2$ \cite{cator2011shape} and for \emph{Exponential Last Passage Percolation} on $\mathbb Z^2$ \cite{ferrari2005competition}. It is worth noting that, except for a few remarkable cases, the lack of isotropy in lattice models poses a significant barrier to obtaining such a property. Nevertheless, \emph{Busemann functions} have recently emerged as a powerful alternative for analyzing infinite geodesics in first and last passage percolation models in the planar setting, although many results still rely on the notorious \emph{global curvature} hypothesis, which is conjectured to hold in general but remains unproven (see, e.g., \cite{damron2014busemann, rassoul2018busemann}).

Overall, straightness remains the only general approach for analyzing infinite branches beyond planarity, yet results in such settings are very scarce. The RST is well understood in dimension $2$, and while it is natural to expect that the results extend to higher dimensions, in practice, the existing arguments rely crucially on the planar structure and break down once it is lost, leaving the problem open, as we explain in more detail below.\\

To understand why existing approaches fail to extend beyond planarity, let us first emphasize that studying the RST presents significant challenges due to the complex geometrical dependencies that arise in its construction. For each $x\in\R^d$ and $r\in\R_+$, let us denote by $B(x, r)$ the \emph{open ball} of \emph{radius} $r$ centered at $x$, and define the \emph{lens}
    \[B^0(x,r)\coloneqq B(x,r)\cap B(0,\|x\|).\]
In the RST, for any $x\in\N$, the presence of the edge $(x,\Psi(x))$ imposes that $B^0(x,\|\Psi(x)-x\|)$ contains no points of $\N$. Since this lens typically overlaps $B(0,\|\Psi(x)\|)$, the edge $(\Psi(x),\Psi\circ\Psi(x))$ is strongly constrained by $(x,\Psi(x))$. An illustration of this dependency phenomenon is presented in Figure \ref{fig_dependencies}.

\begin{figure}[h]
    \centering
    \includegraphics[width=0.3\linewidth]{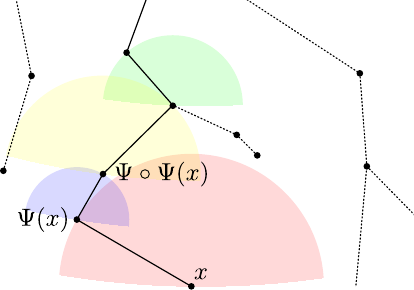}
    \caption{Visualization of the geometrical constraints imposed during the construction of the RST in dimension $2$. In translucent colors are represented the radial cropped balls that must be empty of points.}
    \label{fig_dependencies}
\end{figure}

This intricate correlation structure prevents one from directly exhibiting a Markovian property when iteratively exploring the forward path from a vertex. This limitation greatly complicates the analysis of how paths deviate from straight lines, which is crucial for proving straightness. To circumvent this difficulty in dimension~$2$, \cite{BacceliBordenave} introduced the Directed Spanning Forest (DSF), a closely related model that can be seen as the local limit of the RST around $-te_d$ as $t\to\infty$, and which is more tractable, in particular due to its translation invariance. Then, by analyzing DSF paths using tools from queuing theory and Markov chains, they were able to exhibit a regenerative structure and obtain useful control via classical concentration inequalities. While the relationship between the DSF and the RST through a local limit does not directly allow the deviation bounds to be transferred, they were able to exploit planarity to construct a global coupling between RST and DSF paths (see \cite[Theorem 5.1]{BacceliBordenave}), effectively enabling the derivation of straightness. In higher dimensions, however, this coupling with the DSF completely breaks down, forcing one to work directly with the RST, where translation invariance is absent and the radial structure presents significant challenges, particularly outside the planar setting and given the complex dependency structure.

In this work, we introduce a novel approach for establishing the straightness of the RST beyond planar settings, building instead on advances in the independent study of the DSF itself by \cite{coupier20212d}, recently extended to arbitrary dimensions in \cite{coalVSdim}. While the DSF was originally used as an auxiliary model in the planar setting through coupling arguments, here we do not rely on a direct link between the two models. Rather, we transfer and adapt analytical techniques developed independently for the DSF to the RST. Specifically, a key contribution of our work is the construction of a renewal-type decomposition for RST paths. Then, by applying classical concentration inequalities, we show that RST paths cannot deviate far from straight lines and derive straightness in arbitrary dimensions. It is worth noting that, while the strategy is natural in principle, its implementation involves delicate considerations.

\subsection*{Structure of the paper}

In Section~\ref{section_exploration}, we introduce the main notations by defining the \emph{exploration process} and the \emph{history sets} used to encode the geometric constraints of the RST, and we state the key deviation estimate we aim to establish (Theorem~\ref{thm_deviation_control}), along with a proof showing how the main theorem follows from it. In Section~\ref{section_good_steps}, we focus on controlling the history sets, proving that their size repeatedly returns below a specific threshold (Theorem~\ref{thm_good_steps}). In Section~\ref{section_symmetrization_decomposition}, we leverage this control for a symmetrization decomposition (Proposition~\ref{prop_symmetrization_decomposition}) with suitable fluctuation properties (Proposition~\ref{prop_fluctuation_control}), allowing us to conclude by the proof of Theorem~\ref{thm_deviation_control} in Section~\ref{section_conclusion} using classical concentration inequalities.

\section{The exploration process}

\label{section_exploration}

In this section, we introduce the main notation as well as the central process that explores iteratively forward paths of the RST, with history sets to encode dependencies. From now on, we fix a deterministic starting point $\pi_0\in\R^d$. We define $(\pi_n)_{n\geq1}$ by induction with
    \[\pi_n\coloneqq\Psi(\pi_{n-1})\]
for all $n\geq 1$. We also denote $R_n\coloneqq \|\pi_n\|$ for each $n\geq0$. We call $(\pi_n)_{n\geq0}$ the \emph{exploration process}. As discussed in the introduction, it is highly non-Markovian: each pair $(\pi_n, \pi_{n+1})$ for $n\geq0$ imposes strong geometric constraints on the underlying Poisson configuration, which may in turn influence arbitrarily distant future steps. In order to encode this dependencies, for each $n\geq 0$ we define the \emph{history set}
    \[H_n\coloneqq\bigcup_{k=0}^{n-1}B^0(\pi_k, \|\pi_{k+1}-\pi_k\|),\]
In words, $H_n$ collects all lenses that must be empty of Poisson points after the first $n$ exploration steps are revealed. For each $n\geq 0$, we also define the sigma field
    \[\F_n\coloneqq\sigma[\N\setminus B(0, R_n),~\N \cap\overline H_n],\]
where $\overline H_n$ denotes the closure of $H_n$. The family $(\F_n)_{n\geq 0}$ forms a filtration, to which the sequence $(\pi_n, H_n)_{n\geq 0}$ is adapted. It is worth noting that, for technical reasons to be addressed later, this filtration contains more information than what is strictly necessary.

We can now state the key proposition of this section, which justifies the importance of the
history sets. It formalizes the intuitive idea that after $n\geq 0$ steps of exploration, no information about $\N$ has been revealed in the region $B(0, R_n)\setminus \overline H_n$, and allows to present a crucial characterization of the sequence $(\pi_n, H_n)_{n\geq 0}$. We omit its proof, which is straightforward and technical.

\begin{Proposition}
    \label{prop_resampling}
    Let $\N'$ denote an independent copy of $\N$. For any $n\geq0$, conditionally on $\F_n$, the random variable $\N\cap B(0,R_n)\setminus \overline H_n$ is distributed as $\N'\cap B(0,R_n)\setminus H_n$. In particular, conditionally on $\F_n$, $\pi_{n+1}$ is distributed as $\Psi(\pi_n,\N'\setminus H_n)$.
\end{Proposition}

A notable consequence of Proposition \ref{prop_resampling} is that the exploration process with history sets $(\pi_n, H_n)_{n\geq0}$ is a Markov chain.\\

We now present the main estimate we aim to establish, stated in the following theorem.

\begin{Theorem}
    \label{thm_deviation_control}
    For any $x, u\in\R^d$, define the orthogonal component of $u$ relatively to $x$ by
        \[\mathbf p_{\perp x}(u)\coloneqq u-\frac{x\cdot u}{\|x\|^2}x\]
    whenever $x\neq 0$, and $\mathbf p_{\perp 0}(u)\coloneqq u$ by convention. Then, for all $\varepsilon>0$ there exists $C_\varepsilon=C_\varepsilon(d)>0$ and $c_\varepsilon=c_\varepsilon(d)>0$ such that for all $\pi_0\in\R^d$,
        \[\P\left[\sup_{n\geq 0}\|\mathbf p_{\perp \pi_0}(\pi_n)\|>\|\pi_0\|^{\frac{1}{2}+\varepsilon}\right]\leq C_\varepsilon\exp\left[-c_\varepsilon\|\pi_0\|^{\frac{2}{5}\varepsilon}\right].\]
\end{Theorem}

\begin{figure}[h]
    \centering
    \includegraphics[width=0.5\linewidth]{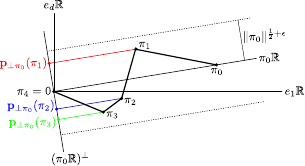}
    \caption{Illustration of Theorem~\ref{thm_deviation_control}.}
    \label{fig_deviation}
\end{figure}

In words, Theorem~\ref{thm_deviation_control} provides control over the \emph{deviation} of RST paths from \emph{straight lines}. An illustration is provided in Figure~\ref{fig_deviation}. The proof of the theorem is deferred to section~\ref{section_conclusion}: establishing this estimate will constitute the core of our work. It is worth mentioning that, in easier settings such as simple random walks, analogous deviation control can be obtained using classical concentration inequalities. In our setting, however, these tools cannot be applied directly due to the complex dependency structure. This is why we will need to construct a suitable renewal-type decomposition.

We now conclude this subsection by deriving our main theorem from Theorem~\ref{thm_deviation_control}. We first establish an elementary lemma, and then proceed to the proof of the theorem.

\begin{Lemma}
    \label{lemma_psi_tail}
    For all $x\in\R^d$ and $t\in\R_+$,
        \[\P(\|\Psi(x)-x\|>t)\leq e^{-(t/2)^d|B(0,1)|}\]
\end{Lemma}

\begin{proof}
    Let $x\in\R^d$. First, note that $\|\Psi(x)-x\|<\|x\|$ by definition of $\Psi$. Hence, if $t\geq\|x\|$,
        \[\P(\|\Psi(x)-x\|>t)=0.\]
    Now assume $t<\|x\|$. Setting $c\coloneqq x(1-\frac{t}{2\|x\|})$ one has
        \begin{equation}
            \label{eq_ball_inclusion_tover2}
            B\left(c,\frac{t}{2}\right)\subset B^0(x, t).
        \end{equation}
    Indeed, if $z\in B(c,\frac{t}{2})$, then $\|z-x\|\leq \|z-c\|+\|c-x\|<\frac{t}{2}+\|x\|\frac{t}{2\|x\|}=t$ and as $\|x\|-\frac{t}{2}\geq 0$,
        \[\|z\|\leq\|z-c\|+\|c\|<\frac{t}{2}+\left(\|x\|-\frac{t}{2}\right)=\|x\|.\]
    Now observe that if $\|\Psi(x)-x\|>t$ then there must be no points of $\N$ in $B^0(x, t)$, and thus one can write
        \begin{equation*}
            \begin{split}
                \P(\|\Psi(x)-x\|>t)&\leq\P[\N\cap B^0(x, t)=\emptyset]\\
                &\leq \P[\N\cap B(c, t/2)=\emptyset]
                =e^{-|B(c, t/2)|}
                =e^{-(t/2)^d|B(0,1)|},
            \end{split}
        \end{equation*}
    where the second inequality follows from \eqref{eq_ball_inclusion_tover2}. The proof is complete.
\end{proof}

\begin{proof}[Proof of Theorem \ref{thm_main}]
    Let us first check that in the RST, with probability one, all vertices have finite in-degree. Using Lemma~\ref{lemma_psi_tail} and the Mecke equation, one has
        \[\E\left[\sum_{x\in\N}\1_{\|\Psi(x)-x\|>\|\pi_0\|^{\frac{1}{2}}}\right]=\int_{\R^d}\P\left[\|\Psi(x)-x\|>\|x\|^{\frac{1}{2}}\right]\mathrm dx\leq\int_{\R^d}\alpha e^{-\beta \|x\|^{\frac{d}{2(d+1)}}}\mathrm dx<\infty,\]
    hence there is almost surely only a finite number of points $x\in\N$ such that $\|\Psi(x)-x\|>\|x\|^{\frac{1}{2}}$. 
    For each $y\in\N$, let $K_y\coloneqq\inf\{t>0:\forall s\geq t,~s^{1/2}<s-\|y\|\}<\infty$. Then one can write
        \[\{x\in\N:\Psi(x)=y\}\subset B(y,\|y\|+K_y)\cup\{x\in\R^d:\|\Psi(x)-x\|>\|x\|^{\frac{1}{2}}\},\]
    Indeed, if $x\in\N$ is such that $\Psi(x)=y$ and $\|x-y\|\geq K_y+\|y\|$, then $\|x\|\geq\|x-y\|-\|y\|\geq K_y$ hence
        \[\|x\|^\frac{1}{2}<\|x\|-\|y\|\leq\|x-y\|=\|\Psi(x)-x\|.\]
    Since $\N\cap B(y, \|y\|+K_y)$ is almost surely finite by local finiteness of $\N$, one deduces that the RST satisfies the finite in-degree property with probability one.

    As the property of being almost surely asymptotically omnidirectional is straightforward for the Poisson point process $\N$, %
    %Let us now turn to the asymptotically omnidirectional property of $\N$. Let $\{\zeta_n\}_{n\geq1}$ be deterministic countable fixed dense subset of $\mathbb S^{d-1}$. For each $m,n,k\geq 1$, we have
    %    \begin{equation*}
    %        \begin{split}
    %            \P\left[\left\{\xi\in\mathbb S^{d-1}:\xi\cdot\zeta_n>1-\frac{1}{k}\right\}\cap\left\{\frac{u}{\|u\|}:u\in\N,~\|u\|\geq m\right\}=\emptyset\right]\qquad&\\
    %            =\P\left(\N\cap\left\{x\in\R^d:\frac{x\cdot\zeta_n}{\|x\|}>1-\frac{1}{k},~\|x\|>m\right\}=\emptyset\right)&=0,
    %        \end{split}
    %    \end{equation*}
    %where one used that the \emph{truncated cone} $\{x\in\R^d:\frac{x\cdot\zeta_n}{\|x\|}>1-\frac{1}{k},~\|x\|>m\}$ has infinite Lebesgue measure to get the last equality. This shows that with probability one, for each $n,k,m\geq 1$,
    %    \[\left\{\xi\in\mathbb S^{d-1}:\xi\cdot\zeta_n>1-\frac{1}{k}\right\}\cap\left\{\frac{u}{\|u\|}:u\in\N,~\|u\|\geq m\right\}\neq\emptyset,\]
    %which implies that $\N$ is almost surely asymptotically omnidirectional.
    it remains to show the straightness property of the RST. To that end, fix $\varepsilon\in(0,\frac{1}{2})$ and observe that using the Mecke equation with Theorem \ref{thm_deviation_control}, denoting
        \[\dev(x)\coloneqq\sup_{n\geq0}\|\mathbf p_{\perp x}(\Psi^n(x))\|\]
    for every $x\in\R^d$, where $\Psi^n$ refers to the $n$-th composition of $\Psi$, one can write
        \begin{equation*}
            \begin{split}
                \E\left[\sum_{x\in\N}\1_{\dev(x)>\|x\|^{\frac{1}{2}+\varepsilon}}\right]&=\int_{\R^d}\P\left[\sup_{n\geq0}\|\mathbf p_{\perp \pi_0}(\pi_n)\|>\|\pi_0\|^{\frac{1}{2}+\varepsilon}\right]\mathrm d\pi_0\\
                &\leq\int_{\R^d}C_\varepsilon\exp\left(-c_\varepsilon\|\pi_0\|^{\frac{2}{5}\varepsilon}\right)\mathrm d\pi_0<\infty.
            \end{split}
        \end{equation*}
    Therefore, there is almost surely only finitely many points $x\in\N$ that satisfies $\dev(x)>\|x\|^{\frac{1}{2}+\varepsilon}$. Combining with the previous arguments, one deduces that there exists almost surely $M_\varepsilon>0$ such that for all $x\in\N$ with $\|x\|>M_\varepsilon$,
        \[\dev(x)\leq\|x\|^{\frac{1}{2}+\varepsilon}\quad\text{and}\quad\|\Psi(x)-x\|\leq\|x\|^\frac{1}{2}.\]
    Finally, applying the arguments of \cite[Theorem 2.6 and Lemma 2.7]{howard2001geodesics} with the exponents $\frac{3}{4}$ replaced by $\frac{1}{2}$, one obtains that the RST is almost surely straight for the function 
        \[f:t\mapsto t^{-\frac{1}{2}+\varepsilon},\]
    which completes the proof.
\end{proof}

\section{Good steps}

\label{section_good_steps}

As the history sets encode the geometrical constraints imposed during the exploration process, it is
crucial to control their size to understand how far the interactions can extend into the future. To this end, for each $n\geq0$ we define
    \[r_n\coloneqq\inf(\{R_n\}\cup\{\|x\|:x\in H_n\})\quad\text{and}\quad L_n\coloneqq r_n-R_n.\]
In words, $L_n$ represent the radial width of $H_n$ in the ball $B(0, R_n)$. The smaller this quantity, the weaker the influence of $H_n$ on the next step of the process. The goal of this section is to show the following.

\begin{Theorem}
    \label{thm_good_steps}
    There exist constants $\lambda,\kappa,c_\tau,C_\tau>0$, depending only on $d$, such that the following holds. For any $\pi_0\in\R^d$, defining the stopping times $\Theta$ and $(\tau_n)_{n\geq0}$ by
        \begin{equation}
            \label{def_theta}
            \Theta\coloneqq\inf\{n\geq0:R_n<1+\kappa~\textnormal{ or }~(L_n)^{d+1}>\lambda R_n\},
        \end{equation}
    with $\tau_0\coloneqq0$ and, for each $n\geq0$,
        \begin{equation}
            \label{def_tau}
            \tau_{n+1}\coloneqq\Theta\wedge\inf\{i\in\mathbb N:i>\tau_n,~L_i<\kappa,~R_{\tau_n}-R_i\geq\kappa+1\},
        \end{equation}
    one has: for all $n,k\geq0$,
        \[\P(\tau_{n+1}-\tau_n>k~|~\F_{\tau_n})\leq C_\tau e^{-c_\tau k}.\]
    The exploration process steps associated with the sequence of times $(\tau_n)_{n\geq0}$ are then called \emph{good steps}.
\end{Theorem}

To establish this theorem, we first control from above the growth of the non-decreasing sequence $(-r_n)_{n\geq0}$, then control from below that of $(-R_n)_{n\geq0}$, and finally combine the two estimates to conclude.

\subsection{Upper control on the growth of $(-r_n)_{n\geq0}$}

In this subsection, we focus on establishing an upper bound on the growth of the sequence $(-r_n)_{n\geq0}$. First, observe that for each $n\geq0$ and $h\in\R_+$, $r_n-r_{n+1}>h$ implies by construction that $B^0(\pi_n, r_n+h)\cap\N=\emptyset$. In particular, the set
    \[B^0(\pi_n, L_n+h)\setminus \overline H_n\]
must be empty of Poisson points. Since this region is included in $B(0,R_n)\setminus\overline H_n$, applying Proposition \ref{prop_resampling}, we get that conditionally on $\F_n$, the number of points in $B^0(\pi_n, r_n+h)\setminus \overline H_n$ is a Poisson random variable of parameter $|B^0(\pi_n, L_n+h)\setminus H_n|$, where $|\cdot|$ stands for the Lebesgue measure in $\R^d$, hence
    \begin{equation}
        \label{eq_proba_overshoot}
        \P[r_n-r_{n+1}>h~|~\F_n]\leq\P[\N\cap B^0(\pi_n, L_n+h)\setminus \overline H_n~|~\F_n]=e^{-|B^0(\pi_n, L_n+h)\setminus H_n|}.
    \end{equation}
The strategy is now to control $|B^0(\pi_n, r_n+h) \setminus H_n|$ independently of the shape of $H_n$. This will be accomplished via the following key lemma.

\begin{Lemma}
    \label{lemma_empty_ball}
    Let $\ell\in[0,1]$, $\rho>0$ and $c\in\R^d$ be such that
        \[\|c\|\geq 1,\quad B(c,\rho)\cap B(0,1-\ell)=\emptyset\quad\text{and}\quad-e_d\notin B(c,\rho).\]
    Then, setting $\alpha_\ell\coloneqq \sqrt{2-\ell}-1\in[0,1]$,
        \[B(-[1-\ell]e_d, \alpha_\ell\ell)\cap B(c, \rho)=\emptyset.\]
    In particular, setting $x_\ell\coloneqq-(1-\ell+\frac{\alpha_\ell}{2}\ell)e_d$, one has
        \[B\Big(x_\ell,\frac{\alpha_\ell}{2}\ell\Big)\subset B^0(-e_d,\ell)\setminus B(c, \rho).\]
\end{Lemma}

\begin{figure}[h]
    \centering
    \includegraphics[width=0.5\linewidth]{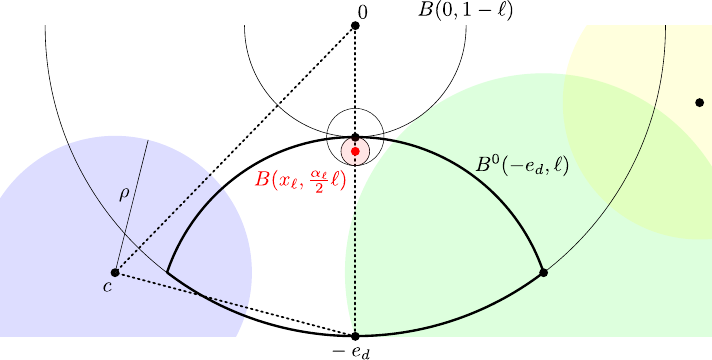}
    \caption{Illustration of Lemma \ref{lemma_empty_ball}. In green the worst setup possible for $B(c,\rho)$, with $\|c\|=1$ and $\rho=\ell=\|c+e_d\|$.}
    \label{fig_empty_ball}
\end{figure}

\begin{proof}
    Observe that since $(1-\ell)\frac{c}{\|c\|}\in\partial B(0, 1-\ell)$ and $B(c,\rho)\cap B(0,1-\ell)=\emptyset$, one must has 
        \[\ell+\|c\|-1=\left\|(1-\ell)\frac{c}{\|c\|}-c\right\|\geq\rho.\]
    Since $-e_d\notin B(c,\rho)$, it implies
        \begin{equation}
            \label{eq_rho}
            \rho\leq\min\{\ell+\|c\|-1,\|c+e_d\|\}.
        \end{equation}
    In a first step, one wants to lower bound $\|c+(1-\ell)e_d\|^2-\rho^2$. Using the polarization identity, one has
        \begin{equation}
            \label{eq_center+}
            \|c+(1-\ell)e_d\|^2=(1-\ell)^2+\|c\|^2+2(1-\ell)c\cdot e_d,
        \end{equation}
    and similarly
        \begin{equation}
            \label{eq_center}
            \|c+e_d\|^2=1+\|c\|^2+2c\cdot e_d.
        \end{equation}
    Using \eqref{eq_center}, one gets
        \begin{equation*}
            \begin{split}
                \|c+e_d\|\leq \ell+\|c\|-1\iff-2c\cdot e_d&\geq 1+\|c\|^2-(\ell+\|c\|-1)^2\\
                &=1-(1-\ell)^2+2(1-\ell)\|c\|\\
                &=\ell(2-\ell)+2(1-\ell)\|c\|\\
                &\eqqcolon s.
            \end{split}
        \end{equation*}
    Let us now use the critical value $s$ to separate two cases. If $-2c\cdot e_d\geq s$, using \eqref{eq_center+} and \eqref{eq_center}, one obtains
        \begin{equation*}
            \begin{split}
                \|c+(1-\ell)e_d\|^2-\rho^2&\geq\|c+(1-\ell)e_d\|^2-\min\{\ell+\|c\|-1,\|c+e_d\|\}^2\\
                &=\|c+(1-\ell)e_d\|^2-\|c+e_d\|^2\\
                &=(1-\ell)^2-1-2\ell c\cdot e_d\\
                &\geq \ell(s+\ell-2)\\
                &=\ell(1-\ell)(\ell+2\|c\|-2),
            \end{split}
        \end{equation*}
    where the first inequality uses \eqref{eq_rho}. In the other case, that is $-2c\cdot e_d<s$, one gets similarly
        \begin{equation*}
            \begin{split}
                \|c+(1-\ell)e_d\|^2-\rho^2&\geq\|c+(1-\ell)e_d\|^2-(\ell+\|c\|-1)^2\\
                &=(1-\ell)^2+\|c\|^2-(\ell+\|c\|-1)^2+2(1-\ell)c\cdot e_d\\
                &\geq(1-\ell)(2\|c\|-s)\\
                &=\ell(1-\ell)(\ell+2\|c\|-2).
            \end{split}
        \end{equation*}
    Summing up, in both cases,
        \begin{equation}
            \label{eq_drho2}
            \|c+(1-\ell)e_d\|^2-\rho^2\geq \ell(1-\ell)(\ell+2\delta)\quad\text{where}\quad\delta\coloneqq\|c\|-1\geq0.
        \end{equation}
    Now, using \eqref{eq_drho2}, one can write
        \begin{equation}
            \label{eq_drho}
            \begin{split}
                \|c+(1-\ell)e_d\|-\rho&\geq\sqrt{\rho^2+\ell(1-\ell)(\ell+2\delta)}-\rho\\
                &\geq\sqrt{(\ell+\delta)^2+\ell(1-\ell)(\ell+2\delta)}-\ell-\delta\\
                &=\ell\sqrt{\left(1+\frac{\delta}{\ell}\right)^2+(1-\ell)\left(1+\frac{2\delta}{\ell}\right)}-\ell-\delta\\
                &=\ell\sqrt{2-\ell+\frac{2\delta}{\ell}(2-\ell)+\frac{\delta^2}{\ell^2}}-\ell-\delta\\
                &\geq\ell\sqrt{2-\ell+2\frac{\delta}{\ell}\sqrt{2-\ell}+\frac{\delta^2}{\ell^2}}-\ell-\delta\\
                &=\ell\sqrt{\left(\sqrt{2-\ell}+\frac{\delta}{\ell}\right)^2}-\ell-\delta\\
                &=\ell\sqrt{2-\ell}-\ell\\
                &=\alpha_\ell\ell,
            \end{split}
        \end{equation}
    where the second inequality follows from concavity of $t\mapsto\sqrt t$ since $\rho\leq \ell+\|c\|-1=\ell+\delta$, and the last one follows from the fact that $t\geq\sqrt{t}$ for all $t\geq 1$, and $2-\ell\geq 1$. Now let $z\in B(-[1-\ell]e_d,\alpha_\ell \ell)$. Using the triangle inequality and \eqref{eq_drho}, one gets
        \[\|z-c\|\geq\|c+(1-\ell)e_d\|-\|z+(1-\ell)e_d\|\geq(\alpha_\ell\ell+\rho)-\alpha_\ell\ell=\rho,\]
    hence $z\notin B(c, \rho)$. This shows 
        \begin{equation}
            \label{eq_emptyball}
            B(-[1-\ell]e_d,\alpha_\ell \ell)\cap B(c,\rho)=\emptyset.
        \end{equation}
    Let us conclude the proof by checking the last part of the result. Let $z\in B(x_\ell,\frac{\alpha_\ell}{2}\ell)$. First, observe that
        \[\|z\|\leq\|z-x_\ell\|+\|x_\ell\|<\frac{\alpha_\ell}{2}\ell+\left(1-\ell+\frac{\alpha_\ell}{2}\ell\right)=1-\ell(1-\alpha_\ell)\leq 1,\]
    where the last inequality uses that $\alpha_\ell\leq 1$. This shows that $z\in B(0,1)$. One also has
        \[\|z+e_d\|\leq\|z-x_\ell\|+\|x_\ell+e_d\|<\frac{\alpha_\ell}{2}\ell+\ell\left(1-\frac{\alpha_\ell}{2}\right)=\ell,\]
    implying $z\in B(-e_d,\ell)$. This shows $z\in B^0(-e_d,\ell)$. Additionally, one can write
        \[\|z+(1-\ell)e_d\|\leq\|z-x_\ell\|+\|x_\ell+(1-\ell)e_d\|<\frac{\alpha_\ell}{2}\ell+\frac{\alpha_\ell}{2}\ell=\alpha_\ell\ell,\]
    implying $z\in B(-(1-\ell)e_d,\alpha_\ell\ell)$. Finally, \eqref{eq_emptyball} gives that $z\notin B(c,\rho)$. The proof is complete.
\end{proof}

\begin{Remark}
    The factor $\alpha_\ell$ goes to $\sqrt2-1$ as $\ell\to0$, that is when the curvature of $\partial B(0, 1)$ become negligible. This constant is exactly the one obtained for the DSF in \cite[Lemma 3.6]{coalVSdim}. Note that Lemma~\ref{lemma_empty_ball} is only asymptotically sharp.
\end{Remark}

We now use the previous lemma to derive the following proposition.

\begin{Proposition}
    \label{prop_upper_control_proba}
    There exists $c_{\mathrm r}=c_{\mathrm r}(d)>0$ such that for all $n\geq 0$ and $h\in\R_+$,
        \[\1_{L_n\leq\frac{1}{2}R_n}\P(r_n-r_{n+1}>h~|~\F_n)\leq\exp({-c_{\mathrm r}\max\{h, L\}^d})\]
\end{Proposition}

\begin{proof}
    Let $n\geq 0$ and $h\in\R_+$. Let us work conditionally on $\F_n$ with $R_n>0$. First, observe that by isotropy of the model, one can assume without loss of generality that $\pi_n=-R_ne_d$. Then, setting
        \[\ell\coloneqq\frac{L_n}{R_n}\in[0,1],\]
    since by definition, $H_n$ consist of an union of lenses $B^0(R_nc,R_n\rho)$ with $\|c\|\geq1$ and $B^0(c,\rho)\cap B(0, \ell)=\emptyset$, one can apply Lemma \ref{lemma_empty_ball} up to a scaling by $\frac{1}{R_n}$ to get that
        \[B\left(R_nx_\ell,\frac{\alpha_\ell}{2}L_n\right)\subset B^0(\pi_n,L_n)\setminus H_n.\]
    Therefore, 
        \[|B^0(\pi_n,L_n+h)\setminus H_n|\geq\left|B\left(R_nx_\ell,\frac{\alpha_\ell}{2}L_n\right)\right|=\left(\frac{\alpha_\ell}{2}L_n\right)^d|B(0,1)|.\]
    Injecting in \eqref{eq_proba_overshoot}, it comes
        \begin{equation}
            \label{eq_proba_overshoot_Ln}
            \begin{split}
                \1_{L_n\leq\frac{1}{2}R_n}\P(r_n-r_{n+1}>h~|~\F_n)&\leq \1_{L_n\leq\frac{1}{2}R_n}\exp\left[-\left(\frac{\alpha_\ell}{2}L_n\right)^d|B(0,1)|\right]\\&\leq \exp\left[-\left(\frac{\alpha_{1/2}}{2}\right)^d|B(0,1)|(L_n)^d\right],
            \end{split}
        \end{equation}
    where the last inequality uses that $\alpha_\ell\geq\alpha_{1/2}$ as $\ell\leq\frac{1}{2}$ when $L_n\leq\frac{1}{2}R_n$ and $t\mapsto \alpha_t$ is non-increasing. It remains to get a bound with respect to $h$. Since $r_n-r_{n+1}\leq r_n$, one can assume $h\leq r_n$. Set
        \[c\coloneqq-\left(r_n-\frac{h}{2}\right)e_d.\]
    If $z\in B(c,\frac{h}{2})$, then
        \[\|z\|\leq\|z-c\|+\|c\|<\frac{h}{2}+\left(r_n-\frac{h}{2}\right)=r_n,\]
    implying $z\in B(0,r_n)$ and in particular $z\in B(0,R_n)\setminus H_n$ by definition of $r_n$. Furthermore,
        \[\|z-\pi_n\|=\|z+R_ne_d\|\leq\|z-c\|+\|c+R_ne_d\|\leq\frac{h}{2}+\left(L_n+\frac{h}{2}\right)=L_n+h,\]
    which gives $z\in B(\pi_n, L_n+h)$. This shows
        \[B\left(c,\frac{h}{2}\right)\subset B^0(\pi_n,L_n+h)\setminus H_n,\]
    and thus
        \[|B^0(\pi_n,L_n+h)\setminus H_n|\geq\left(\frac{h}{2}\right)^d.\]
    Injecting in \eqref{eq_proba_overshoot} gives
        \begin{equation}
            \label{eq_proba_overshoot_h}
            \P(r_n-r_{n+1}>h~|~\F_n)\leq e^{(h/2)^d}.
        \end{equation}
    Finally, putting \eqref{eq_proba_overshoot_Ln} and \eqref{eq_proba_overshoot_h} together gives the result setting $c_\mathrm r\coloneq\min\{1, (\frac{\alpha_{1/2}}2)^d|B(0,1)|\}$.
\end{proof}

We conclude the subsection by establishing the following key moment bound from the previous proposition.

\begin{Proposition}
    \label{prop_upper_control_moment}
    There exists a non-increasing function $G:\R_+\to\R$ with $\lim_{L\to\infty} G(L)=1$ such that for all $n\geq 0$,
        \[\1_{L_n\leq\frac{1}{2}R_n}\E[e^{2(r_n-r_{n+1})}~|~\F_n]\leq G(L_n).\]
\end{Proposition}

\begin{proof}
    Set
        \[G:\R_+\to\R,~L\mapsto1+2\int_0^\infty e^{2t-c_{\mathrm r}\max\{t, L\}^d}\mathrm dt.\]
    Since $G(0)<\infty$ and the integrand in the definition of $G(L)$ decreases to $0$ point-wise as $L\to\infty$, the dominated convergence theorem gives that $G(L)$ decreases to $1$. Now, observe that using Fubini's theorem, as $r_n-r_{n+1}$ is a positive random variable, one can write
        \[\E[e^{2(r_n-r_{n+1})}~|~\F_n]=1+2\int_0^\infty e^{2t}\P(r_n-r_{n+1}>t~|~\F_n)\mathrm dt.\]
    Finally, using Proposition \ref{prop_upper_control_proba}, it comes
        \[\1_{L_n\leq\frac{1}{2}R_n}\E[e^{2(r_n-r_{n+1})}~|~\F_n]\leq1+2\int_0^\infty e^{2t}e^{-c_{\mathrm r}\max\{t,L_n\}^d}\mathrm dt=G(L_n).\]
    The proof is complete.
\end{proof}

\subsection{Lower control on the growth of $(-R_n)_{n\geq 0}$}

In this subsection, we focus on establishing a lower bound on the growth of the sequence $(-R_n)_{n\geq 0}$. To do so, we need to control the \emph{progress} toward the origin made at every step $n\geq0$, that is, the difference $R_n-R_{n+1}$. Our approach relies on a stochastic domination technique introduced in \cite{coalVSdim} to handle an analogous problem for the DSF. To apply it in our setting, we will need to control how far the RST deviates from the idealized situation where no curvature is visible, that is, infinitely far from the origin. To that end, we introduce the following lemma. An illustration is provided in Figure \ref{fig_flatness_default}.

\begin{Lemma}
    \label{lemma_flatness_default}
    Let $c\in\R^d$, $\rho\in\R_+$ and $\ell\in[0, \frac{1}{2}]$ be such that
        \[\|c\|\geq 1,\quad B(c,\rho)\cap B(0,1-\ell)=\emptyset\quad\text{and}\quad \overline B(c,\rho)\cap \overline B(-e_d,\ell)\neq\emptyset.\]
    Then,
        \[1+c\cdot e_d\leq2\ell^2.\]
\end{Lemma}

\begin{figure}[h]
    \centering
    \includegraphics[width=0.6\linewidth]{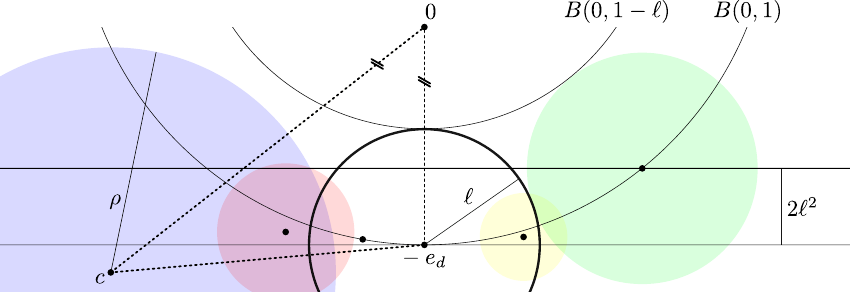}
    \caption{Illustration of Lemma \ref{lemma_flatness_default}. In translucent colors different possible settings for the ball $B(c, r)$ satisfying the required assumptions. The thickest line represents the highest level possible for $c$ on the $e_d$ coordinate, given as the result of the lemma.}
    \label{fig_flatness_default}
\end{figure}

\begin{proof}
    Observe that since $(1-\ell)\frac{c}{\|c\|}\in\partial B(0, 1-\ell)$ and $B(c,\rho)\cap B(0,1-\ell)=\emptyset$, one must have
        \begin{equation}
            \label{eq_rho_flat}
            \ell+\|c\|-1=\left\|(1-\ell)\frac{c}{\|c\|}-c\right\|\geq\rho.
        \end{equation}
    Since $\overline B(-e_d,\ell)\cap \overline B(c,\rho)\neq\emptyset$, there exists $x\in \overline B(-e_d,\ell)\cap \overline B(c,\rho)$ and it follows that
        \begin{equation}
            \label{eq_ced}
            \|c+e_d\|\leq\|c-x\|+\|x+e_d\|\leq\rho+\ell\leq2\ell+\|c\|-1,
        \end{equation}
    where the last inequality uses \eqref{eq_rho_flat}. Now, using the polarization identity, one gets
        \begin{equation*}
            \begin{split}
                -2e_d\cdot c&=1+\|c\|^2-\|c+e_d\|^2\\
                &\geq1+\|c\|^2-(2\ell+\|c\|-1)^2\\
                &=1+\|c\|^2-(\|c\|-1)^2-4\ell^2-4\ell(\|c\|-1)\\
                &=2\|c\|-4\ell(\|c\|-1)-4\ell^2\\
                &=2+2(1-2\ell)(\|c\|-1)-4\ell^2\\
                &\geq2-4\ell^2,
            \end{split}
        \end{equation*}
    where the first inequality follows from \eqref{eq_ced} and the last line uses that $\|c\|\geq 1$ and $\ell\leq\frac{1}{2}$. It follows that $1+c\cdot e_d\leq2\ell^2$. The proof is complete.
\end{proof}

\begin{Remark}
    Lemma \ref{lemma_flatness_default} is sharp. In the case $\|c\|=1$, $\|c+e_d\|=2\ell$ and $\rho=\ell$, illustrated in green on Figure \ref{fig_flatness_default}, all the inequalities of the proof are equalities, leading to $1+c\cdot e_d=2\ell^2$.
\end{Remark}

We will also need the following lemma, which establishes a relation between \emph{radial} and \emph{vertical} progress.

\begin{Lemma}
    \label{lemma_radial_upward_progress}
    Fix $\rho\in[0,1]$ and $h\in[\frac{1}{2},1]$. Let $x\in B^0(-e_d, \rho)$ be such that
        \[1+x\cdot e_d\geq h\rho,\]
    then
        \[1-\|x\|\geq\left(h-\frac{1}{2}\right)\rho.\]
\end{Lemma}

\begin{proof}
    One can write
        \begin{equation}
            \label{eq_x_polarization}
            \begin{split}
                \rho^2>\|x+e_d\|^2&=\|x\|^2+1+2x\cdot e_d\\
                &=\|x\|^2-1+2(1+x\cdot e_d)\\
                &\geq\|x\|^2-1+2h\rho\\
                &\geq-2(1-\|x\|)+2h\rho,
            \end{split}
        \end{equation}
    where the first inequality uses $x\in B(-e_d,\rho)$, the middle line follows from \eqref{eq_Yed}, and the last one is obtained via the mean value theorem applied to the function $z\mapsto z^2$ with $\|x\|\leq 1$ as $x\in B(0,1)$. It follows from \eqref{eq_x_polarization} that
        \[1-\|x\|\geq\left(h-\frac{1}{2}\rho\right)\rho\geq\left(h-\frac{1}{2}\right)\rho,\]
    where the last inequality uses that $\rho\in[0,1]$ and $h\geq\frac{1}{2}$. The proof is complete.
\end{proof}

Let $\mathcal H_0$ denote the collection of \emph{normalized history sets}, namely all subsets of $\mathbb R^d$ that can be expressed as a union of lenses avoiding $-e_d$, whose centers lie at a distance of at least $1$ from the origin. The following proposition builds upon the two previous lemmas and applies the stochastic domination approach of \cite[Section 3.2.1]{coalVSdim} to obtain suitable control on the progress.

\begin{Proposition}
    \label{prop_progress}
    There exist $\delta=\delta(d)>0$, $p=p(d)>0$ and $\lambda=\lambda(d)\in(0,\frac{1}{2})$ such that for all $R\geq L\geq 1$ with $L^{d+1}\leq \lambda R$ and for each normalized history set $H\in\mathcal H_0$ that satisfies $(RH)\cap B(0, R-L)=\emptyset$, we have
        \[\P(R-\|\Psi(-Re_d, \N\setminus RH)\|\geq \delta)\geq p.\]
\end{Proposition}

\begin{figure}[h]
    \centering
    \includegraphics[width=0.7\linewidth]{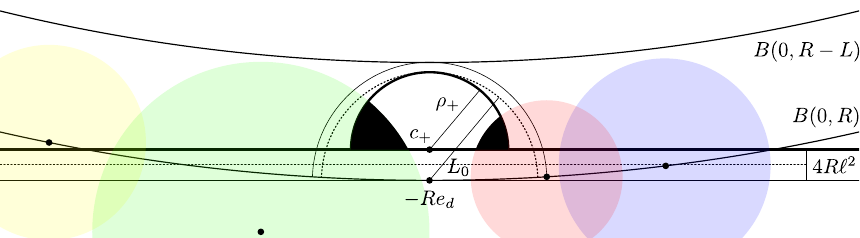}
    \caption{Illustration behind the proof of Proposition \ref{prop_progress}. The translucent balls represent the normalized history $H$ and in bold black is represented the idealized DSF configuration exhibited in the argument.}
    \label{fig_progress}
\end{figure}

\begin{proof}
    The idea is to exhibit with an idealized DSF configuration in the RST setting, as illustrated in Figure~\ref{fig_progress}, in order for existing tools to be applied.
    
    \textbf{step 1.} Let us begin by setting the problem and defining some useful constants. Set
        \[\lambda\coloneqq\frac{1}{4d}\left(\frac{\alpha_{1/2}}{2}\right)^d>0.\]
    Let $R\geq L\geq 1$ and $H\in\mathcal H_0$. Assume $L^{d+1}\leq \lambda R$ and $(RH)\cap B(0, R-L)=\emptyset$. Since $L\geq 1$ and $\alpha_{1/2}\leq1$, one has
        \[\ell\coloneqq\frac{L}{R}\leq\frac{L^{d+1}}{R}\leq \lambda \leq\frac{1}{2}.\]
    As $H$ consists of a union of lenses whose centers lie at a distance greater $1$ from the origin and do not contain $-e_d$, applying Lemma \ref{lemma_empty_ball} up to scaling by $\frac{1}{R}$, one gets
        \[B\left(Rx_\ell, \frac{\alpha_\ell}{2}L\right)\subset B^0(-Re_d, L)\setminus RH,\]
    which implies
        \begin{equation}
            \label{eq_area_L0_up}
            \begin{split}
                |B^0(-Re_d, L)\setminus RH|&\geq \left|B\left(0,\frac{\alpha_\ell}{2}L\right)\right|\geq \left|B\left(0,\frac{\alpha_{1/2}}{2}\right)\right|,
            \end{split}
        \end{equation}
    where one used $L\geq 1$ and $\alpha_{\ell}\geq\alpha_{1/2}$ as $\ell\leq1/2$ and $t\mapsto\alpha_t$ is non-increasing. One also has
        \begin{equation}
            \label{eq_area_L0_down}
            \left|B^0\left(-Re_d,\frac{\alpha_{1/2}}{2}\right)\setminus RH\right|\leq\left|B\left(0,\frac{\alpha_{1/2}}{2}\right)\right|.
        \end{equation}
    Therefore, as $t\mapsto|B^0(-Re_d, t)\setminus RH|$ is continuous, applying the intermediate value theorem with \eqref{eq_area_L0_down} and \eqref{eq_area_L0_up}, there exists $L_0\in[\frac{\alpha_{1/2}}{2}, L]$ such that
        \begin{equation*}
            \label{eq_L0_definition}
            |B^0(-Re_d, L_0)\setminus RH|=\left|B\left(0,\frac{\alpha_{1/2}}{2}\right)\right|\eqqcolon \mu.
        \end{equation*}
    \textbf{step 2.} Let us now exhibit an idealized DSF configuration within the RST setting. The statement and verification of the sought properties will be conducted in a later step. Set
        \[c_+\coloneqq -R(1-4\ell^2)e_d\quad\text{and}\quad \rho_+\coloneqq L_0-4R\ell^2\in\R.\]
    Note that since $\alpha_{1/2}\leq 1\leq L$ and $d\geq2$, $4R\ell^2=\frac{4L^2}{R}\leq\frac{4L^{d+1}}{R}\leq 4\lambda =\frac{1}{d}\left(\frac{\alpha_{1/2}}{2}\right)^d\leq \frac{\alpha_{1/2}}{4}$ and thus
        \begin{equation}
            \label{eq_rho+}
            \rho_+\geq\frac{\alpha_{1/2}}{2}-\frac{\alpha_{1/2}}{4}=\frac{\alpha_{1/2}}{4}>0.
        \end{equation}
    For any $c\in\R^d$ and $r\geq0$, define $\mathbb H^-(c\cdot e_d)\coloneqq\{y\in\R^d:y\cdot e_d\leq c\cdot e_d\}$ and $B^+(c, r)\coloneq B(c, r)\setminus\mathbb H^-(c\cdot e_d)$. Let us first show that $B^+(c_+,\rho_+)\subset B^0(-Re_d, L_0)$. Let $x\in B^+(c_+,\rho_+)$. One has
        \[\|x+Re_d\|\leq\|x+c_+\|+\|c_++Re_d\|<(L_0-4R\ell^2)+4R\ell^2=L_0,\]
    which gives $x\in B(-Re_d,L_0)$. One has also $x\in B(-Re_d,R\ell)$ as $L_0\leq L=R\ell$, implying $\overline B(x, 0)\cap\overline B(-Re_d,R\ell)=\{x\}\neq \emptyset$. Since $B(x, 0)\cap B(0,R(1-\ell))=\emptyset$ and 
        \[R+x\cdot e_d>R+c_+\cdot e_d=4R\ell^2>2R\ell^2\]
    by definition of $x\in B^+(c_+, \rho_+)\subset\R^d\setminus\mathbb H^-(c_+\cdot e_d)$, scaling by $\frac{1}{R}$ one can apply the contrapositive of Lemma \ref{lemma_flatness_default} to get that $\|x\|<R$. This shows that
        \begin{equation}
            \label{eq_B+subsetB0}
            B_+\coloneqq B^+(c_+,\rho_+)\subset B(c_+,\rho_+)\cap B(0, R)=B^0(-Re_d, L_0).
        \end{equation}
    The set $B_+\setminus RH$ will be our candidate for the idealized DSF configuration. 
    
    \textbf{step 3.} Let us now show that on an event of probability bounded away from zero independently of the configuration, $\Psi(-Re_d,\N\setminus RH)$ is uniformly distributed in $B_+\setminus RH$. To that end, first observe that the complementary set $B_-\coloneqq B^0(-Re_d, L_0)\setminus B_+$ satisfies
        \begin{equation*}
            \begin{split}
                |B_-|&\leq|B^+(-Re_d, L_0)\setminus B^+(c_+,\rho_+)|\\
                &=({L_0}^d-{\rho_+}^d)|B^+(0,1)|\\
                &\leq d{L_0}^{d-1}(L_0-\rho_+)|B^+(0,1)|,
            \end{split}
        \end{equation*}
    where the first inequality is obtained by inclusion of the underlying sets, the equality uses \eqref{eq_B+subsetB0}, and the last line follows from the mean value theorem applied to the function $z\mapsto z^d$. As $L_0\leq L$ and $L_0-\rho_+= 4R\ell^2=4\frac{L^2}{R}$, it follows that
        \begin{equation}
            \label{eq_volB-}
            \begin{split}
                |B_-|\leq 4d\frac{L^{d+1}}{R}|B^+(0,1)|
                &\leq 4d\lambda |B^+(0,1)|\\
                &=\left(\frac{\alpha_{1/2}}{2}\right)^d|B^+(0,1)|=\frac{1}{2}\left|B\left(0, \frac{\alpha_{1/2}}{2}\right)\right|=\frac{\mu}{2}.
            \end{split}
        \end{equation}
    In particular, this implies that
        \begin{equation}
            \label{eq_B+_volume}
            |B_+\setminus RH|=|B^0(-Re_d, L_0)\setminus RH|-|B_-\setminus RH|\geq \mu-|B_-|\geq\frac{\mu}{2}>0.
        \end{equation}
    Now, consider the event
        \[A\coloneqq\{\#\N\cap B_+\setminus RH=1\}\cap\{\#\N\cap B_-\setminus RH=0\}.\]
    Since $\{B_+,B_-\}$ forms a partition of $B^0(-Re_d,L_0)$, $|B_+\setminus RH|+|B_-\setminus RH|=|B^0(-Re_d,L_0)\setminus RH|=\mu$, and by properties of the Poisson point process,
        \begin{equation}
            \label{eq_proba_A}
            \P(A)=|B_+\setminus RH|e^{-|B_+\setminus RH|}e^{-|B_-\setminus RH|}
            =|B_+\setminus RH|e^{-\mu}
            \geq\frac{1}{2}\mu e^{-\mu}>0.
        \end{equation}
    Finally, observe that when $A$ occurs, there is only one element in $\N\cap B(0, R)\setminus RH$ within distance $L_0$ from $-Re_d$, which is the unique point of $\mathcal N\cap B_+\setminus RH$. This implies that conditional on $A$, $\Psi(-Re_d, \N\setminus RH)$ is uniformly distributed on $B_+\setminus RH$.

    \textbf{step 4.} Now, let us verify that $B_+\setminus RH$ indeed corresponds to an idealized DSF configuration in the sense that it can be expressed as the half-ball $B_+$ minus a union of balls avoiding $c_+$, whose centers lie in $\mathbb H^-(c_+\cdot e_d)$. To that end, first observe that by definition of $H\in\mathcal H_0$, one can write
        \[RH=\bigcup_{i\in I}B^0(c_i,\rho_i),\]
    where $I$ is some index set and for each $i\in I$, $c_i\in\R^d$ with $\|c_i\|\geq R$ as well as $-Re_d\notin B(c_i, \rho_i)$. Set $I_0\coloneqq\{i\in I:B(c_i,\rho_i)\cap B_+\neq\emptyset\}$ and
        \[H'\coloneqq\bigcup_{i\in I_0}B(c_i,\rho_i).\]
    By construction,
        \[B_+\setminus RH=B^+(c_+,\rho_+)\setminus H'.\]
    Let $i\in I_0$. The definition of $I_0$ with \eqref{eq_B+subsetB0} and $L_0\leq L$ gives $B(c_i, \rho_i)\cap B(-Re_d, L)\neq\emptyset$. Hence, as $B(c_i,\rho_i)\cap B(0,R-L)\subset (RH)\cap B(0, R-L)=\emptyset$, scaling by $\frac{1}{R}$ and applying Lemma \ref{lemma_flatness_default} gives
        \begin{equation}
            \label{eq_ci_ed}
            R+c_i\cdot e_d\leq 2R\ell^2.
        \end{equation}
    Since $R+c_+\cdot e_d=4R\ell^2\geq2R\ell^2$, it implies
        \begin{equation*}
            %\label{eq_ci_below_c+}
            c_i\in\mathbb H^-(c_+\cdot e_d).
        \end{equation*}
    Additionally, one has
        \begin{equation*}
            \begin{split}
                \|c_i-c_+\|^2-\|c_i+Re_d\|^2&=(R+c_i\cdot e_d-4R\ell^2)^2-(R+c_i\cdot e_d)^2\\
                &=8R\ell^2(2R\ell^2-R-c_i\cdot e_d)\\
                &\geq 0,
            \end{split}
        \end{equation*}
    where the inequality follows from \eqref{eq_ci_ed}. Since $\|c_i+Re_d\|\geq\rho_i$ from $-Re_d\notin B(c_i,\rho_i)$, one deduces $\|c_i-c_+\|^2\geq\|c_i+Re_d\|^2\geq{\rho_i}^2$, and thus $c_+\notin B(c_i,\rho_i)$. This shows that $H'$ consists of a union of balls avoiding $c_+$ whose centers lie in $\mathbb H^-(c_+\cdot e_d)$. This allow us to apply \cite[Proposition 3.2]{coalVSdim} up to a translation and a scaling: denoting by $Y$ a uniform point in $B^+(c_+,\rho_+)\setminus H'$, and by $U^\emptyset$ a uniform point of $B^+(0,1)$, one gets
        \begin{equation}
            \label{eq_sto_dom}
            \frac{1}{\rho_+}(Y-c_+)\cdot e_d\succeq_{\mathrm{sto}} U^\emptyset\cdot e_d,
        \end{equation}
    that is $\P(Y\cdot e_d\geq c_+\cdot e_d+h\rho_+)\geq\P(U^\emptyset\geq h)$ for all $h\in[0,1]$.

    \textbf{step 5.} Let us now put everything together to conclude. First, observe that if $(Y-c_+)\cdot e_d\geq\frac{3}{4}\rho_+$, then
        \begin{equation}
            \label{eq_Yed}
            \begin{split}
                R+Y\cdot e_d=4R\ell^2+(Y-c_+)\cdot e_d
                &\geq 4R\ell^2+\frac{3}{4}\rho_+\\
                &=4R\ell^2+\frac{3}{4}(L_0-4R\ell^2)\\
                &\geq\frac{3}{4}4R\ell^2+\frac{3}{4}(L_0-4R\ell^2)
                =\frac{3}{4}L_0.
            \end{split}
        \end{equation}
    Then, applying Lemma \ref{lemma_radial_upward_progress} up to scaling by $\frac{1}{R}$, one gets from \eqref{eq_Yed} that whenever $(Y-c_+)\cdot e_d\geq\frac{3}{4}\rho_+$,
        \begin{equation}
            \label{eq_Rprogress}
            R-\|Y\|\geq\left(\frac{3}{4}-\frac{1}{2}\right)L_0\geq\frac{\alpha_{1/2}}{8}\eqqcolon\delta>0.
        \end{equation}
    Putting everything together, one gets that
        \begin{equation}
            \begin{split}
                \P\left[R-\|\Psi(-Re_d,\N\setminus RH)\|\geq \delta \right]&\geq\P\left[A, R-\|\Psi(-Re_d,\N\setminus RH)\|\geq \delta \right]\\
                &=\P(A)\P(R-\|Y\|\geq\delta )\\
                &\geq\frac{\mu}{2}e^{-\mu}\P\left[(Y-c_+)\cdot e_d\geq\frac{3}{4}\rho_+\right]\\
                &\geq\frac{\mu}{2}e^{-\mu}\P\left(U^\emptyset\cdot e_d\geq\frac{3}{4}\right)\\
                &\eqqcolon p_0>0,
            \end{split}
        \end{equation}
    where the first equality uses that conditional on $A$, $\Psi(-Re_d,\N\setminus RH)$ is distributed as $Y$, the second inequality follows from \eqref{eq_Rprogress} as well as \eqref{eq_proba_A}, the second to last line is obtained via \eqref{eq_sto_dom} and the final inequality comes from the fact that for any $h<1$, $P(U^\emptyset\cdot e_d\geq h)>0$. The proof is complete.
\end{proof}

For now on, $\lambda$ is fixed and given by Proposition \ref{prop_progress}. We conclude the subsection by deriving an exponential moment bound from the previous result.

\begin{Proposition}
    \label{prop_lower_control}
    There exists $a=a(d)>0$ such that for all $n\geq 0$,
        \[\1_{1\leq(L_n)^{d+1}\leq\lambda R_n}\E[e^{2(R_{n+1}-R_n)}~|~\F_n]\leq a<1.\]
\end{Proposition}

\begin{proof}
    Let $n\geq0$. Let us reason conditionally on $\F_n$ with $R_n>0$. By isotropy of the model, one can assume without loss of generality that $\pi_n=-R_ne_d$ and thus $\frac{1}{R_n}H_n\in\mathcal H_0$. Since from Proposition \ref{prop_resampling}, conditionally on $\F_n$, $R_n-R_{n+1}$ is distributed as
        \[R_n-\|\Psi(-R_n e_d,\N'\setminus H_n)\|,\]
    where $\N'$ denotes an independent copy of $\N$, then applying Proposition \ref{prop_progress} yields
        \[\1_{1\leq (L_n)^{d+1}\leq\lambda R_n}\P(R_n-R_{n+1}\geq\delta~|~\F_n)\geq p.\]
    As $R_n-R_{n+1}$ is a positive random variable, one deduces that
        \[\1_{1\leq(L_n)^{d+1}\leq\lambda R_n}\E[e^{2(R_{n+1}-R_n)}~|~\F_n]\leq pe^{-2\delta}+(1-p)\eqqcolon a<1.\]
    The proof is complete.
\end{proof}

\subsection{Proof of Theorem \ref{thm_good_steps}}

In this subsection, we establish Theorem \ref{thm_good_steps}. We start with the following proposition, which leverages the bound obtained in the previous subsections to derive a strong exponential control on history size increments.

\begin{Proposition}
    \label{prop_exp_moment_q}
    There exists $\kappa=\kappa(d)>1$ and $q=q(d)>0$ such that
        \[\1_{(L_n)^{d+1}\leq\lambda R_n}\1_{L_n\geq \kappa}\E[e^{L_{n+1}-L_n}~|~\F_n]\leq q<1.\]
\end{Proposition}

\begin{proof}
    Using a Cauchy-Schwarz inequality, one can write
        \begin{equation*}
            \begin{split}
                \1_{1\leq (L_n)^{d+1}\leq\lambda R_n}\E[e^{L_{n+1}-L_n}~|~\F_n]^2
                &\leq\1_{1\leq(L_n)^{d+1}\leq\lambda R_n}\E[e^{2(R_{n+1}-R_n)}~|~\F_n]\E[e^{2(r_n-r_{n+1})}~|~\F_n]\\
                &\leq\1_{1\leq(L_n)^{d+1}\leq\lambda R_n}\E[e^{2(R_{n+1}-R_n)}~|~\F_n]G(L_n),
            \end{split}
        \end{equation*}
    where the last inequality follows from Proposition \ref{prop_upper_control_moment} since $1\leq(L_n)^{d+1}\leq\lambda R_n$ implies that $L_n\leq (L_n)^{d+1}\leq\lambda R_n\leq\frac{1}{2}R_n$. Then using Proposition \ref{prop_lower_control}, for all $\kappa>1$ one gets
        \[\1_{(L_n)^{d+1}\leq\lambda R_n}\1_{L_n\geq\kappa}\E[e^{L_{n+1}-L_n}~|~\F_n]^2\leq aG(\kappa)\to a<1\quad\text{as $\kappa\to\infty$}.\]
    Hence, fixing $q\in(\sqrt a, 1)$, the convergence above ensures that one can fix $\kappa>1$ big enough to obtain the result. The proof is complete.
\end{proof}

For now on, $\kappa>1$ is fixed and given by Proposition \ref{prop_exp_moment_q}. Recall the definition of the stopping time $\Theta$ in \eqref{def_theta}. Note that since $\N$ is locally finite, $\pi_n=0$ for all $n\geq0$ big enough, so that $\Theta$ is almost surely finite. In order to prove Theorem \ref{thm_good_steps}, for each $n\geq 0$, we introduce the stopping time
    \[T_n\coloneqq\Theta\wedge\inf\{j>n:R_n-R_j\geq1+\kappa\}.\]
We start by the following lemma.

\begin{Lemma}
    \label{lemma_T_n}
    For all $n,k\geq0$, one has
        \[\P(T_n-n>k~|~\F_n)\leq e^{2(\kappa+1)}a^k\]
\end{Lemma}

\begin{proof}
    One can write
        \begin{equation*}
            \begin{split}
                \P(T_n-n>k~|~\F_n)
                &=\P(n+k<\Theta,~R_n-R_{n+k}<\kappa+1~|~\F_n)\\
                &\leq\E\left[\1_{R_n-R_{n+k}<\kappa+1}\prod_{i=0}^{k-1}\1_{n+i<\Theta}~\middle|~\F_n\right]\\
                &\leq e^{2(\kappa+1)}\E\left[\prod_{i=0}^{k-1}\1_{i<\Theta}e^{2(R_{n+i+1}-R_{n+i})}~\middle|~\F_n\right]\\
                &\leq e^{2(\kappa+1)}a^k,
            \end{split}
        \end{equation*}
    where the second inequality is obtained via the bound $\1_{x<y}\leq e^{2(y-x)}$ for all $(x, y)\in\R^2$ as well as a telescopic product, and the last line follows from Proposition \ref{prop_lower_control} since for each $i\geq 0$, $i<\Theta$ implies $1\leq(L_i)^{d+1}\leq\lambda R_i$ by definition. The proof is complete.
\end{proof}

For each $n\geq0$, we also introduce the stopping time
    \[\sigma_n\coloneq\Theta\wedge\inf\{j\geq n:L_j<\kappa,~R_n-R_j\geq\kappa+1\},\]
and establish the following.

\begin{Lemma}
    \label{lemma_sigma_n}
    For any $n\geq0$ and $\ell\geq k\geq 0$, one has
        \[\P(T_n-n\leq k,~\sigma_n-n>\ell~|~\F_n)\leq 1_{n<\Theta}e^{L_n-\kappa}G(0)^kq^{\ell-k}.\]
\end{Lemma}

\begin{proof}
    First, one can write
        \begin{equation*}
            \begin{split}
                \1_{T_n-n\leq k}\1_{\sigma_n-n>\ell}
                &\leq\1_{n+\ell<\Theta}\prod_{i=0}^{\ell-k}\1_{L_{n+k+i}\geq\kappa}\\
                &\leq\1_{n+\ell<\Theta}e^{L_{n+\ell}-\kappa}\prod_{i=0}^{\ell-k-1}\1_{L_{n+k+i}\geq\kappa}\\
                &\leq\1_{n+k<\Theta}e^{L_{n+k}-\kappa}\prod_{i=0}^{\ell-k-1}\1_{n+k+i<\Theta}\1_{L_{n+k+i}\geq\kappa}e^{L_{n+k+i+1}-L_{n+k+i}},
            \end{split}
        \end{equation*}
    where the first inequality uses the fact that $T_n\leq j<\sigma_n$ implies $L_j\geq\kappa$ by construction, the second follows from the bound $\1_{x\geq y}\leq e^{x-y}$ for all $(x,y)\in\R^2$, and the last line is obtained via a telescopic product. As for each $j\geq 0$, $j<\Theta$ implies $(L_j)^{d+1}\leq\lambda R_j$, taking the expectation conditional on $\F_n$ and using Proposition \ref{prop_exp_moment_q}, it comes
        \begin{equation*}
            \begin{split}
                \P(T_n-n\leq k,~\sigma_n-n>\ell~|~\F_n)&\leq\E[\1_{n+k<\Theta}e^{L_{n+k}-\kappa}~|~\F_n]q^{\ell-k}\\
                &=e^{L_n-\kappa}\E\left[\prod_{i=0}^{k-1}1_{n+i<\Theta}e^{L_{n+i+1}-L_{n+i}}~\middle|~\F_n\right]q^{\ell-k}\\
                &\leq e^{L_n-\kappa}\E\left[\prod_{i=0}^{k-1}1_{n+i<\Theta}e^{2(r_{n+i}-r_{n+i+1})}~\middle|~\F_n\right]q^{\ell-k}\\
                &\leq 1_{n<\Theta}e^{L_n-\kappa}G(0)^kq^{\ell-k},
            \end{split}
        \end{equation*}
    where the equality uses a telescopic product, the second to last inequality uses that $L_{j+1}-L_j\leq r_j-r_{j+1}\leq 2(r_j-r_{j+1})$ for all $j\geq0$, and the last line is obtained via Proposition \ref{prop_upper_control_moment} as $1\leq L_j\leq (L_j)^{d+1}\leq\frac{1}{2}R_j$ for each $j\geq 0$ with $j<\Theta$. The proof is complete.
\end{proof}

We can now conclude the proof of Theorem \ref{thm_good_steps}, recalling the definition of the stopping times $(\tau_n)_{n\geq0}$ in \eqref{def_tau}.

\begin{proof}[Proof of Theorem \ref{thm_good_steps}]
    Let $n, k\geq 0$. By definition, one has
        \[\tau_{n+1}=\sigma_{\tau_n}.\]
    Fix an integer $A\geq1$ big enough so that $G(0)q^A<1$. From Lemma \ref{lemma_sigma_n}, one has
        \begin{equation}
            \label{eq_sigma1}
            \P[T_{\tau_n}-\tau_n\leq k,~\sigma_{\tau_n}-\tau_n>(A+1)k~|~\F_{\tau_n}]\leq \1_{\tau_n<\Theta}e^{L_{\tau_n}-\kappa}G(0)^kq^{Ak}\leq G(0)^kq^{Ak},
        \end{equation}
    where the last inequality uses that $L_{\tau_n}<\kappa$ when $\tau_n<\Theta$. From Lemma \ref{lemma_T_n}, one has also
        \begin{equation}
            \label{eq_sigma2}
            \P(T_{\tau_n}-\tau_n>k~|~\F_{\tau_n})\leq e^{2(\kappa+1)}a^k.
        \end{equation}
    Putting together \eqref{eq_sigma1} and \eqref{eq_sigma2} using a union bound, it comes
        \[\P[\tau_{n+1}-\tau_n>(A+1)k~|~\F_{\tau_n}]=\P[\sigma_{\tau_n}-\tau_n>(A+1)k~|~\F_{\tau_n}]\leq G(0)^kq^{Ak}+e^{2(\kappa+1)}a^k,\]
    and the result follows as both $a<1$ and $G(0)q^A<1$. The proof is complete.
\end{proof}

\section{Symmetrization decomposition}

\label{section_symmetrization_decomposition}

In this section, we build on the stopping times $(\tau_n)_{n\geq0}$ constructed earlier to define a suitable renewal-type decomposition of the exploration process and control the fluctuations of the path between its successive blocks. The strategy is inspired by \cite{coalVSdim}, but delicate adaptations were required to handle the radial setting.

\subsection{Construction}

In this subsection, we introduce the decomposition and establish the main properties it was intended to satisfy. Let us start with some notation. For all $n\geq0$, we define
    \[\pi_n^\uparrow\coloneqq\pi_n\left(1-\frac{\kappa}{R_n}\right)\]
if $R_n\neq 0$ and $\pi_n^\uparrow\coloneqq0$ otherwise. In words, when $R_n\geq\kappa$, $\pi_n^\uparrow$ is obtained by moving $\pi_n$ radially $\kappa$ units toward the origin. For all $n\geq 0$, we define the \emph{pseudo-renewal} event
    \[Q_{\tau_n+1}\coloneqq\{\#\N\cap B^0(\pi_{\tau_n},\kappa+1)=\#\N\cap B^0(\pi_{\tau_n}^\uparrow,1)=1\}.\]
In words, $Q_{\tau_n+1}$ occurs when there is only one point of $\N$ in $B^0(\pi_{\tau_n},\kappa+1)$ that is actually in $B^0(\pi_{\tau_n}^\uparrow,1)$. An illustration of the realization of this event is presented in Figure~\ref{fig_renewal}. The key intuition is as follows. When $\tau_n<\Theta$, we have $L_{\tau_n}<\kappa$ and thus $H_{\tau_n}\cap B(0,\|\pi_n^\uparrow\|)=\emptyset$. If, in addition, the event $Q_{\tau_n+1}$ occurs, then by construction
    \[\Psi(\pi_{\tau_n})=\Psi(\pi_{\tau_n}^\uparrow),\]
effectively allowing the exploration process to restart from $\pi_n^\uparrow$ with an empty history set and a future distribution that depends only on the position of this point, with some symmetry property. The remainder of this subsection is devoted to making this argument rigorous and constructing the corresponding decomposition of the exploration process.

\begin{figure}[h]
    \centering
    \includegraphics[width=0.6\linewidth]{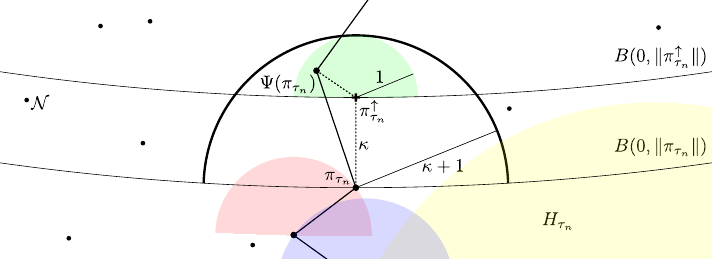}
    \caption{Illustration of a realization of the pseudo-renewal event $Q_{\tau_n+1}$ with $\tau_n<\Theta$.}
    \label{fig_renewal}
\end{figure}

To construct the decomposition, we begin by defining the first occurrence of the terminal time $\Theta$ in the sequence $(\tau_n)_{n\geq0}$ as
    \[I_\Theta\coloneqq\inf\{i\geq0:\tau_i=\Theta\}.\]
Then we define an auxiliary sequence of indices $(\eta_k)_{k\geq 0}$ inductively by
    \[\eta_0\coloneqq 0\quad\text{and}\quad\eta_{k+1}\coloneqq I_\Theta\wedge\inf\{i>\eta_k:Q_{\tau_i+1}~\text{occurs},~\tau_{i+1}<\Theta\}\]
for each $k\geq0$. Finally, we define the pseudo-renewal times $(w_n)_{n\geq0}$ by
    \[w_n\coloneqq\tau_{\eta_n}\]
for all $n\geq0$. Intuitively, $(w_i)_{i\ge 0}$ marks the times at which pseudo-renewal events occur, up to some technical adjustments. It is this sequence that defines the desired decomposition. Before establishing its key property, we need to introduce some enhanced $\sigma$-fields as $(w_n)_{n\geq0}$ are not stopping times with respect to the filtration $(\F_n)_{n\geq0}$. First, for each $i\geq0$ we set
    \[\S_i\coloneqq\sigma(\F_{\tau_i}, Q_{\tau_i+1}\cap\{\tau_{i+1}<\Theta\}).\]
The following lemma ensures that this sequence defines a filtration that is thinner than $(\F_{\tau_i})_{i\geq0}$.

\begin{Lemma}
    \label{lemma_filtration}
    For each $i\geq0$, we have $\F_{\tau_i}\subset\S_i\subset\F_{\tau_{i+1}}$. In particular, $(\S_i)_{i\geq0}$ forms a filtration.
\end{Lemma}

\begin{proof}
    Let $i\geq0$. By definition of $\F_{\tau_i}$, it is enough to check that $Q_{\tau_i+1}\cap\{\tau_{i+1}<\Theta\}$ is $\F_{\tau_{i+1}}$-measurable. The event $Q_{\tau_i+1}$ depends only on the configuration of $\N$ outside the set $B(0, R_{\tau_i}-\kappa-1)$. Moreover, when $\tau_{i+1}<\Theta$, one has $R_{\tau_{i+1}}\geq R_{\tau_i}+\kappa+1$. Since by definition $\sigma(\N\setminus B(0, R_{\tau_{i+1}}))\subset\F_{\tau_{i+1}}$, it follows that $Q_{\tau_i+1}\cap\{\tau_{i+1}<\Theta\}$ is indeed $\F_{\tau_{i+1}}$-measurable. This completes the proof.
\end{proof}

Then, since by construction the sequence $(\eta_k)_{k\geq0}$ consists of $(\S_i)_{i\geq0}$ stopping times, for each $n\geq 0$ the sigma-field
    \[\G_n\coloneqq \S_{\eta_n}\]
is well-defined. We can now conclude the subsection by formally designating the \emph{symmetrization decomposition} as an adapted process and establishing its key symmetry property in the following proposition.

\begin{Proposition}
    \label{prop_symmetrization_decomposition}
    We call \emph{symmetrization decomposition} the process $(\pi_{w_n})_{n\geq0}$. It is adapted to the filtration $(\G_n)_{n\geq0}$. Moreover, for each $n\geq0$, the distribution of $\mathbf p_{\perp \pi_{w_n}}(\pi_{w_{n+1}})$ conditional on $\G_n$ is symmetric.
\end{Proposition}

\begin{figure}[h]
    \centering
    \includegraphics[width=0.6\linewidth]{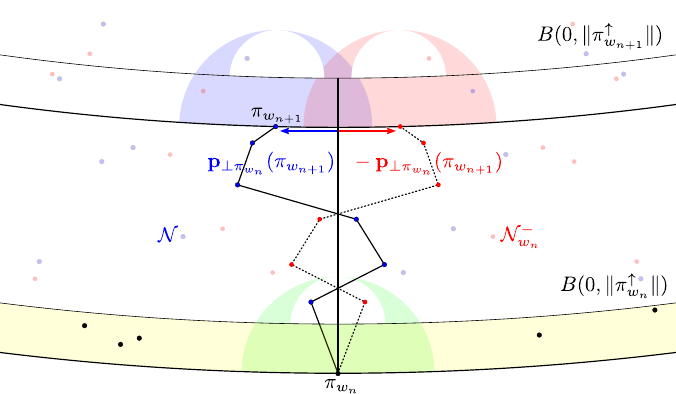}
    \caption{Illustration of the argument behind the proof of Proposition \ref{prop_symmetrization_decomposition}. The vector $p_{\perp\pi_{w_n}}(\pi_{w_{n+1}})$ is represented as a blue arrow, its opposite in red.}
    \label{fig_symmetry}
\end{figure}

\begin{proof}
    Let $n\geq 0$. By construction, $\pi_{w_n}$ is $\G_n$-measurable. To establish the symmetry property, let us construct a transformation of the underlying point configuration, illustrated in Figure~\ref{fig_symmetry}, that preserves its law and yields opposite values of $\mathbf p_{\perp \pi_{w_n}}(\pi_{w_{n+1}})$. For any $x\in\R^d\setminus\{0\}$, define the mapping
        \[\Phi_x:\R^d\to\R^d,~y\mapsto\begin{cases}
            y-2\mathbf p_{\perp x}(y)&\text{if $\|y\|<\|x\|$}\\
            y&\text{otherwise.}
        \end{cases}\]
    Intuitively, $\Phi_x$ acts as the symmetry through the axis generated by $x$ inside $B(0,\|x\|)$, leaving all points outside this ball unchanged. This application clearly preserves the Lebesgue measure, hence $\Phi_x(\N)$ and $\N$ are identically distributed. Let us now couple the random variable $\N$ with its transformed version by setting
        \[\N_{w_n}^-\coloneqq\begin{cases}
            \Phi_{\pi_0}(\N)&\text{if $n=0$}\\
            \N&\text{if $n>0$ and $w_n=\Theta$}\\
            \Phi_{\pi_{w_n}^\uparrow}(\N)&\text{otherwise.}
        \end{cases}\]
    Using Proposition \ref{prop_resampling}, we obtain that, conditionally on $\G_n$, $\N$ and $\N_{w_n}^-$ are identically distributed. By construction, running the exploration process from the same initial point $\pi_0$ on both $\N$ and $\N_{w_n}^-$ produces opposite values of $\mathbf p_{\perp \pi_{w_n}}(\pi_{w_{n+1}})$. This implies that, conditionally on $\G_n$, $\mathbf p_{\perp \pi_{w_n}}(\pi_{w_{n+1}})$ is symmetrically distributed. The proof is complete.
\end{proof}

\subsection{Fluctuation control}

In this subsection, we leverage the control on the spacing between consecutive good steps provided by Theorem \ref{thm_good_steps} as well as properties of the pseudo-renewal event to derive a bound on the fluctuation of the RST path between consecutive pseudo-renewal times. The goal is to show the following proposition.

\begin{Proposition}
    \label{prop_fluctuation_control}
    There exists $c_\Sigma=c_\Sigma(d)>0$ and $C_\Sigma=C_\Sigma(d)>0$ such that for $t>0$ and $\pi_0\in\R^d$,
        \[\P\left[\sup_{n\geq 0}\sum_{k=w_n}^{w_{n+1}-1}\|\pi_{k+1}-\pi_k\|>t\right]\leq C_\Sigma\|\pi_0\|e^{-c_\Sigma\sqrt t}.\]
\end{Proposition}

We start with the following two lemmas.

\begin{Lemma}
    \label{lemma_i_theta}
    The inequality $I_\Theta\leq\|\pi_0\|$ holds almost surely. Furthermore, $w_n=\Theta$ for all $n\geq I_\Theta$.
\end{Lemma}

\begin{proof}
    By construction, $R_{\tau_i}-R_{\tau_{i+1}}\geq(\kappa+1)\1_{i+1<I_\Theta}$ for all $i\geq 0$. As $\Theta<\infty$ almost surely, summing over all $i\geq0$ one obtains that
        \[\|\pi_0\|=R_0\geq R_0-R_\Theta\geq(\kappa+1)I_\Theta,\]
and it follows that $I_\Theta\leq\|\pi_0\|$. Now let $n\geq0$ and assume $w_n<\Theta$. Then by definition, $\eta_n<I_\Theta$. Since $\eta_{k+1}-\eta_k\geq \1_{\eta_k<I_\Theta}$ for all $k\geq1$, one deduces that
        \[n=\sum_{k=0}^{n-1}\1_{\eta_k<I_\Theta}\leq\sum_{k=0}^{n-1}(\eta_{k+1}-\eta_k)=\eta_n<I_\Theta.\]
    This shows that $w_n<\Theta$ implies $n<I_\Theta$. The proof is complete.
\end{proof}

\begin{Lemma}
    \label{lemma_renewal_proba}
    There exists $b=b(d)>0$ such that for all $n\geq 0$,
        \[\P(Q_{\tau_n+1}~|~\F_{\tau_n})\geq b\1_{n<I_\theta}.\]
    In particular, for each $n\geq0$, $\P(Q_{\tau_{n+1}+1}~|~\S_n)\geq b\1_{n<I_\theta}$. In words, conditional on $\F_{\tau_n}$ and provided $n<I_\Theta$, the pseudo-renewal event $Q_{\tau_n+1}$ occurs with probability bounded away from $0$.
\end{Lemma}

\begin{proof}
    Let $n\geq0$. Set
        \[U_{\tau_n}\coloneqq B^0(\pi_{\tau_n},\kappa+1)\setminus[H_n\cup B^0(\pi_{\tau_n}^\uparrow, 1)]\quad\text{and}\quad V_{\tau_n}\coloneqq B^0(\pi_{\tau_n}^\uparrow, 1).\]
    Using Proposition \ref{prop_resampling} and denoting by $\N'$ an independent copy of $\N$, one can write
        \begin{equation}
            \label{eq_proba_renewal}
            \begin{split}
                \1_{i<I_\theta}\P(Q_{\tau_i+1}~|~\F_{\tau_i})&=\1_{i<I_\theta}\P(\#\N'\cap U_{\tau_n}=0, ~\#\N'\cap V_{\tau_n}=1~|~\F_{\tau_n})\\
                &=\1_{i<I_\theta}e^{-|U_{\tau_n}|}|V_{\tau_n}|e^{-|V_{\tau_n}|},
            \end{split}
        \end{equation}
    where one used that $U_{\tau_n}$ and $V_{\tau_n}=\emptyset$ are disjoint, as well as the fact that by construction, $H_{\tau_n}\cap B(0,\|\pi_{\tau_n}^\uparrow\|)=\emptyset$ when $n<I_\Theta$. It now remains to exhibit suitable deterministic bounds for $|U_{\tau_n}|$ and $|V_{\tau_n}|$ which do not depend of $n$. Since $U_{\tau_n}\subset B(\pi_{\tau_n},\kappa+1)$ and $V_{\tau_n}\subset B(\pi_{\tau_n}^\uparrow,1)$, one has
        \begin{equation}
            \label{eq_volume_upper}
            |U_{\tau_n}|\leq(\kappa+1)^d|B(0,1)|\quad\text{and}\quad|V_{\tau_n}|\leq|B(0,1)|.
        \end{equation}
    Now, observe that when $n<I_\Theta$ one has $\|\pi_{\tau_n}^\uparrow\|\geq 1$, so that setting
        \[c\coloneqq\pi_{\tau_n}^\uparrow\left(1-\frac{1}{2\|\pi_{\tau_n}^\uparrow\|}\right),\]
    one gets 
        \begin{equation}
            \label{eq_bc12}
            B\left(c,\frac{1}{2}\right)\subset B^0(\pi_{\tau_n}^\uparrow, 1)
        \end{equation}
    Indeed, if $x\in B(c,\frac{1}{2})$, then $\|x\|\leq\|c\|+\|x-c\|<(\|\pi_n^\uparrow\|-\frac{1}{2})+\frac{1}{2}=\|\pi_n^\uparrow\|$ and similarly $\|x-\pi_n^\uparrow\|\leq\|x-c\|+\|c-\pi_n^\uparrow\|<\frac{1}{2}+\frac{1}{2}=1$. From \eqref{eq_bc12} it comes
        \[|V_{\tau_n}|\geq\frac{1}{2^d}|B(0,1)|,\]
    and injecting together with \eqref{eq_volume_upper} in \eqref{eq_proba_renewal}, one deduces that
        \[\P(Q_{\tau_n+1}~|~\F_{\tau_n})\geq b\1_{n<I_\theta}\]
    with
        \[b\coloneqq e^{-(\kappa+1)^d|B(0,1)|}\inf\left\{te^{-t}:t\in\left[\frac{1}{2^d}|B(0,1)|,|B(0,1)|\right]\right\}>0.\]
    This shows the first part of the result. Finally, since $\S_n\subset\F_{\tau_{n+1}}$ from Lemma \ref{lemma_filtration}, one obtains
        \[\1_{n<I_\Theta}\P(Q_{\tau_{n+1}+1}~|~\S_n)=\E[\1_{n<I_\Theta}\P(Q_{\tau_{n+1}+1}~|~\F_{\tau_{n+1}})~|~\S_n]\geq b\1_{n<I_\Theta},\]
    which completes the proof.
\end{proof}

Next, we derive a control on the spacing between consecutive indices $(\eta_k)_{k\geq 0}$ and, in turn, on the corresponding times $(w_n)_{n\geq0}$.

\begin{Lemma}
    \label{lemma_eta}
    For all $n, k\geq0$, one has
        \[\P(\eta_{n+1}-\eta_n>k)\leq(1-b)^{k-1}.\]
    In words, the number of good steps separating between two consecutive pseudo-renewal times admits an exponentially decaying tail.
\end{Lemma}

\begin{proof}
    Let $n\geq 0$ and $k\geq1$. One can write
        \begin{equation*}
            \begin{split}
                \P(\eta_{n+1}-\eta_n>k)&\leq\E\left[\1_{\eta_n+\ell<I_\Theta}\prod_{i=1}^{k-1}\Big(1-\1_{Q_{\tau_{\eta_n+i}+1}}\Big)
                \right]\\
                &\leq\E\left[\prod_{i=1}^{k-1}\1_{\eta_n+i<\Theta}\Big(1-\1_{Q_{\tau_{\eta_n+i}+1}}\Big)
                \right]\\
                &=\E\left[\prod_{i=1}^{k-1}\1_{\eta_n+i<\Theta}(1-\P(Q_{\tau_{\eta_n+i}+1}~|~\S_{\eta_n+i}))
                \right]
                \leq(1-b)^{k-1},
            \end{split}
        \end{equation*}
    where the first line follows from the definition of $\eta_{k+1}$, and the last inequality uses Lemma \ref{lemma_renewal_proba} together with the fact that for each $j\geq0$, $\eta_j$ is a $(\S_n)_{n\geq0}$ stopping time. The proof is complete.
\end{proof}

\begin{Proposition}
    \label{prop_w_spacing}
    There exists $c_{\mathrm w}=c_{\mathrm w}(d)>0$ and $C_{\mathrm w}=C_{\mathrm w}(d)>0$ such that for all $k\geq0$,
        \[\P\left[\sup_{n\geq0}(w_{n+1}-w_n)>k\right]\leq C_{\mathrm w}\|\pi_0\|e^{-c_{\mathrm w} k}.\]
\end{Proposition}

\begin{proof}
    Let $A\in\mathbb N^*$. Since $I_\Theta\leq\|\pi_0\|$ and $w_{n+1}=w_n$ for all $n\geq I_\Theta$ by Lemma \ref{lemma_i_theta}, for any $k\geq0$, one can write
        \begin{equation}
            \label{eq_wn_spacing}
            \begin{split}
                \P\left[\sup_{n\geq0}(w_{n+1}-w_n)>Ak\right]&\leq\P(\exists 0\leq n<I_\Theta~\eta_{n+1}-\eta_n>k)+\P(\exists 0\leq n<I_\Theta~\tau_{n+k}-\tau_n>Ak)\\
                &\leq\|\pi_0\|\left(\sup_{n\geq0}\P(\eta_{n+1}-\eta_n>k)+\sup_{n\geq0}\P(\tau_{n+k}-\tau_n>Ak)\right).
            \end{split}
        \end{equation}
    Theorem \ref{thm_good_steps} implies that there exists $t,M>0$ such that for all $n\geq 0$,
        \[\E[e^{t(\tau_{n+1}-\tau_n)}~|~\F_n]\leq M<\infty.\]
    Then, it comes that for each $n\geq 0$,
        \[\P(\tau_{n+k}-\tau_n>Ak)\leq e^{-tAk}\E[e^{t(\tau_{n+k}-\tau_n)}]\leq e^{-tAk}M^k.\]
    Injecting in \eqref{eq_wn_spacing} together with Lemma \ref{lemma_eta}, one obtains that
        \[\P\left[\sup_{n\geq0}(w_{n+1}-w_n)>Ak\right]\leq\|\pi_0\|\left[(1-b)^{k-1}+(Me^{-tA})^k\right].\]
    Finally, fixing $A$ big enough so that $Ce^{-tA}<1$, the result follows.
\end{proof}

To derive Proposition~\ref{prop_fluctuation_control} from the previous result, we will need the following lemma.

\begin{Lemma}
    \label{lemma_sum_pi_exp_mom}
    There exists $Z\in(1, \infty)$ such that for all $j>i\geq0$ one has
        \[\E\left[\exp\left(\frac{1}{j-i}\sum_{n=i}^{j-1}\|\pi_{n+1}-\pi_n\|\right)~\middle|~\F_i\right]\leq e^{L_i}Z^{j-i}.\]
\end{Lemma}

\begin{proof}
    Let $j>i\geq0$. First, observe that for each $n\geq0$, $(\|\pi_{n+1}-\pi_n\|-L_n)_+=r_n-r_{n+1}$, so that
        \[\|\pi_{n+1}-\pi_n\|\leq L_n+r_n-r_{n+1}.\]
    Summing for $n$ between $i$ and $j-1$, it comes
        \begin{equation}
            \label{eq_Wij}
            \sum_{n=i}^{j-1}\|\pi_{n+1}-\pi_n\|\leq r_i-r_j+\sum_{n=i}^{j-1}L_n.
        \end{equation}
    Now observe that
        \begin{equation}
            \label{eq_sum_L}
            \begin{split}
                \sum_{n=i}^{j-1}L_n&=\sum_{n=i}^{j-1}\left(L_i+\sum_{k=i}^{n-1}(L_{k+1}-L_k)\right)\\&=(j-i)L_i+\sum_{n=i}^{j-1}(j-1-n)(L_{n+1}-L_n)\\
                &\leq(j-i)L_i+(j-i)\sum_{n=i}^{j-1}(r_n-r_{n+1})
                =(j-i)(L_i+r_i-r_j),
            \end{split}
        \end{equation}
    where the inequality uses that $L_{n+1}-L_n\leq r_n-r_{n+1}$ for each $n\geq0$ by definition. Then, injecting \eqref{eq_sum_L} in \eqref{eq_Wij}, it comes
        \[\frac{1}{j-i}\sum_{n=i}^{j-1}\|\pi_{n+1}-\pi_n\|\leq L_i+\left(1+\frac{1}{j-i}\right)(r_i-r_j)\leq L_i+2(r_i-r_j).\]
    Taking the exponential and the expectation, one obtains
        \begin{equation}
            \label{eq_expmoment_sum_dst}
            \E\left[\exp\left(\frac{1}{j-i}\sum_{n=i}^{j-1}\|\pi_{n+1}-\pi_n\|\right)~\middle|~\F_i\right]\leq e^{L_i}\E[e^{2(r_i-r_j)}~|~\F_i].
        \end{equation}
    Finally, observe that Proposition \ref{prop_upper_control_proba} implies that there exists $Z\in(1,\infty)$ such that
        \[\E[e^{2(r_n-r_{n+1})}~|~\F_n]\leq Z\]
    for all $n\geq0$, so that $\E[e^{2(r_i-r_j)}~|~\F_i]\leq Z^{j-i}$, which injected in \eqref{eq_expmoment_sum_dst} concludes the proof.
\end{proof}

We can finally prove the main result of the subsection.

\begin{proof}[Proof of Proposition \ref{prop_fluctuation_control}]
    Let $t, A\in\mathbb N^*$. First, observe that one can write
        \begin{equation}
            \label{eq_p_sumpi}
            \begin{split}
                &\P\left[\sup_{n\geq0}\sum_{k=w_n}^{w_{n+1}-1}\|\pi_{k+1}-\pi_k\|>At^2\right]\\
                &\qquad\leq\P\left[\sup_{n\geq 0}(w_{n+1}-w_n)>t\right]+\P\left[\exists 0\leq n<I_\Theta~\sum_{k=\tau_n}^{\tau_n+t-1}\|\pi_{k+1}-\pi_k\|>At^2\right]\\
                &\qquad\leq\|\pi_0\|\left(C_{\mathrm w} e^{-c_{\mathrm w} t}+\sup_{n\geq0}\P\left[\sum_{k=\tau_n}^{\tau_n+t-1}\|\pi_{k+1}-\pi_k\|>At^2\right]\right),
            \end{split}
        \end{equation}
    where the first inequality follows from a union bound together with the fact that for any $n\geq 0$, $w_n<\Theta$ implies $w_n=\tau_i$ for some $i<I_\Theta$, and the last line is obtained using Proposition \ref{prop_w_spacing} as well as $I_\Theta\leq\|\pi_0\|$ from Lemma \ref{lemma_i_theta}. Now, let $n\geq0$ and observe that using Markov's inequality, one can write
        \begin{equation*}
            \begin{split}
                \P\left[\sum_{k=\tau_n}^{\tau_n+t-1}\|\pi_{k+1}-\pi_k\|>At^2\right]&\leq e^{-At}\E\left[\exp\left(\frac{1}{t}\sum_{k=\tau_n}^{\tau_n+t-1}\|\pi_{k+1}-\pi_k\|\right)\right]\\
                &=e^{-At}\E\left[\1_{\tau_n<\Theta}\E\left[\exp\left(\frac{1}{t}\sum_{k=\tau_n}^{\tau_n+t-1}\|\pi_{k+1}-\pi_k\|\right)~\middle|~\F_{\tau_n}\right]+\1_{\tau_n=\Theta}\right]\\
                &\leq e^{-At}\E[\1_{\tau_n<\Theta}e^{L_{\tau_n}}Z^t+\1_{\tau_n=\Theta}]\leq e^\kappa Z^te^{-At},
            \end{split}
        \end{equation*}
    where $C$ is given by Lemma \ref{lemma_sum_pi_exp_mom} and the last inequality uses that $L_{\tau_n}<\kappa$ when $\tau_n<\Theta$ by definition. Injecting in \eqref{eq_p_sumpi}, it comes
        \[\P\left[\sup_{n\geq0}\sum_{k=w_n}^{w_{n+1}-1}\|\pi_{k+1}-\pi_k\|>At^2\right]\leq\|\pi_0\|\left[C_{\mathrm w} e^{-c_{\mathrm w} t}+e^\kappa (Ze^{-A})^t\right].\]
    Fixing $A$ big enough so that $Ce^{-A}<1$, one derives the result.
\end{proof}

To conclude the subsection, we focus on controlling the fluctuations of the exploration process beyond $\Theta$, which are not captured by the symmetrization decomposition, since $(w_n)_{n\geq0}$ eventually reaches $\Theta$ but then stays stationary. Because the path $(\pi_n)_{n\geq\Theta}$ stays within $B(0, R_\Theta)$, it suffices to suitably bound the tail of the random variable $R_\Theta$ uniformly over $\pi_0\in\R^d$. This is what we accomplish in the following proposition.

\begin{Proposition}
    \label{prop_r_theta}
    Recall that $R_\Theta=\|\pi_\Theta\|$ where $\Theta$ is defined in \eqref{def_theta}. Then, there exists $C_\Theta=C_\Theta(d)>0$ and $c_\Theta=c_\Theta(d)>0$ such that for all $t\geq1$,
        \[\P(R_\Theta>t)\leq C_\Theta e^{-c_\Theta t^{\frac{d}{d+1}}}.\]
\end{Proposition}

\begin{proof}
    Assume $\Theta>1$. Then, by definition of $L_\Theta$, there exist $c\in\N$ with $B^0(c, \|\Psi(c)-c\|)\subset H_\Theta$, $\|c\|\geq R_\Theta$ and
        \begin{equation}
            \label{eq_psi_c_LT}
            R_\Theta-L_\Theta=r_\Theta=\|c\|-\|\Psi(c)-c\|.
        \end{equation}
    Now let $t\geq 1+\kappa$ and assume $R_\Theta>t$. By definition of $\Theta$, since $R_\Theta>t\geq1+\kappa$, one must have
        \[(L_\Theta)^{d+1}>\lambda R_\Theta.\]
    Injecting in \ref{eq_psi_c_LT}, it implies that $c$ satisfies
        \begin{equation}
            \label{eq_psi_c}
            \|\Psi(c)-c\|>\|c\|-R_\Theta+g(R_\Theta)\quad\text{where}\quad g:x\mapsto(\lambda x)^{\frac{1}{d+1}}.
        \end{equation}
    Since $g(x)=o_{x\to\infty}(x)$, $\|c\|\geq R_\Theta$ and $g$ is concave, one can write
        \begin{equation*}
            \begin{split}
                g(\|c\|)-g(R_\Theta)&\leq g(\|c\|-R_\Theta)-g(0)\\&=g(\|c\|-R_\Theta)\\
                &\leq \|c\|-R_\Theta+C_0\quad\text{where}\quad C_0\coloneqq\sup\{g(x)-x:x\geq 0\}\in\R_+.
            \end{split}
        \end{equation*}
    Injecting in \eqref{eq_psi_c} gives
        \begin{equation}
            \label{eq_c_}
            \begin{split}
                \|\Psi(c)-c\|&>\|c\|-R_\Theta+g(R_\Theta)\\
                &\geq g(\|c\|)-g(R_\Theta)-C_0+g(R_\Theta)
                =(\lambda\|c\|)^{\frac{1}{d+1}}-C_0.
            \end{split}
        \end{equation}
    Then, using a union bound and the Mecke equation, one deduces that
        \begin{equation*}
            \begin{split}
                \P(\Theta>1,~R_\Theta>t)&\leq\P\left[\exists c\in\N\setminus B(0,t)~\|\Psi(c)-c\|>(a_0\|c\|)^{\frac{1}{d+1}}-C_0\right]\\
                &\leq\E\left[\sum_{c\in\N\setminus B(0, t)}\1_{\|\Psi(c)-c\|>(a_0\|c\|)^{\frac{1}{d+1}}-C_0}\right]\\
                &=\int_{\R^d\setminus B(0,t)}\P\left[\|\Psi(x)-x\|>(a_0\|x\|)^{\frac{1}{d+1}}-C_0\right]\mathrm dx.
            \end{split}
        \end{equation*}
    Applying Lemma \ref{lemma_psi_tail} in the display above, one can find $C_1,C_2\in(0,\infty)$ such that for all $t\geq 1$,
        \[\P(\Theta>1,~R_T>t)\leq C_1\int_{\R^d\setminus B(0,t)}e^{-C_2\|x\|^{\frac{d}{d+1}}}\mathrm dx=dC_1|B(0,1)|\int_t^\infty r^{d-1}e^{-C_2r^{\frac{d}{d+1}}}\mathrm dr,\]
    where the last equality is obtained via a polar change of variables. Finally, as one has
        \[\P(\Theta=1,~R_\Theta>t)=\P(\|\Psi(\pi_0)-\pi_0\|>t)\leq e^{-(t/2)^d|B(0,1)|}\]
    from Lemma \ref{lemma_psi_tail}, the result follows.
\end{proof}

\section{Proof of Theorem \ref{thm_deviation_control}}

\label{section_conclusion}

We now finally prove Theorem~\ref{thm_deviation_control} using the symmetrization decomposition.

\begin{proof}[Proof of Theorem \ref{thm_deviation_control}]
    For each $n\geq0$, denote by
        \[\Delta_{n+1}\coloneqq \pi_{w_{n+1}}-\pi_{w_n}\]
    the increment of the trajectory between the renewal steps corresponding to $w_n$ and $w_{n+1}$. Let $A>0$. Set
        \[G_A\coloneqq\{\forall n\geq0~\|\Delta_{n+1}\|\leq A\}.\]
    Let $U=\{u_s\}_{1\leq s\leq d-1}$ denote an orthonormal basis of $(\pi_0\R)^\perp$ and $U^\pm\coloneq U\cup -U$. For all $x\in\R^d$, one has
        \[\|\mathbf p_{\perp\pi_0}(x)\|^2=\sum_{u\in U^\pm}\max\{0, x\cdot u\}^2,\]
    hence since $|U^\pm|=2(d-1)$, there exists $u\in U^\pm$ such that $x\cdot u\geq\frac{\|\mathbf p_{\perp\pi_0}(x)\|}{\sqrt{2(d-1)}}$. For all $j>i\geq0$, set
        \[B_u^{i, j}\coloneqq\{\pi_{w_i}\cdot u\leq 0\}\cap\{\forall i<k\leq j,~\pi_{w_k}\cdot u\geq 0\}\]
    Let $t\geq A$. One can then write
        \begin{equation}
            \label{eq_proba_decomposition}
            \begin{split}
                \P\left[G_A,~\sup_{n\geq0}\|\mathbf p_{\perp \pi_0}(\pi_{w_n})\|>t\sqrt{2(d-1)}\right]&\leq\P[G_A,~\exists u\in U^\pm~\exists 1\leq j\leq I_\Theta~\pi_{w_j}\cdot u>t]\\
                &\leq\sum_{\substack{0\leq i<j\leq \lfloor\|\pi_0\|\rfloor\\u\in U^\pm}}\P[G_A,~B_u^{i, j},~\pi_{w_j}\cdot u>t],
            \end{split}
        \end{equation}
    where one used that $I_\Theta\leq\|\pi_0\|$ from Lemma \ref{lemma_i_theta}. Fix $u\in U^\pm$ and $0\leq i<j\leq\lfloor\|\pi_0\|\rfloor$. Let $i<k<j$. Since
        \[\|\pi_{w_k}\|^2\geq\|\pi_{w_{k+1}}\|^2=\|\pi_{w_k}\|^2+\|\Delta_{k+1}\|^2+2\Delta_{k+1}\cdot \pi_{w_k},\]
    one must have $\Delta_{k+1}\cdot \pi_{w_k}\leq 0$. This implies that whenever $\pi_{w_k}\cdot u\geq 0$, one has
        \begin{equation}
            \label{eq_delta_projection}
            \Delta_{k+1}\cdot u=\mathbf p_{\perp \pi_{w_k}}(\Delta_{k+1})\cdot u+\frac{1}{\|\pi_{w_k}\|^2}(\Delta_{k+1}\cdot\pi_{w_k})(\pi_{w_k}\cdot u)\leq\mathbf p_{\perp \pi_{w_k}}(\Delta_{k+1})\cdot u.
        \end{equation}
    Now denote
        \[D_{k+1}^{u,A}\coloneqq\1_{\|\mathbf p_{\perp\pi_{w_k}}(\Delta_{k+1})\|\leq A}\times\mathbf p_{\perp\pi_{w_k}}(\Delta_{k+1})\cdot u\]
    for all $k\geq0$. On the event $G_A\cap B_u^{i, j}$, one has $\pi_{w_i}\cdot u\leq 0$, $\Delta_{i+1}\cdot u\leq\|\Delta_{i+1}\|\leq A$ and $\Delta_{k+1}\cdot u=D_{k+1}^{u, A}$ for each $i+1\leq k\leq j-1$, therefore
    %since $\pi_{w_j}=\pi_{w_i}+\Delta_{i+1}+\sum_{k=i+1}^{j-1}\Delta_{k+1}$, one deduces that on $G_A\cap B_u^{i, j}$,
        \begin{equation*}
                \pi_{w_j}\cdot u=\pi_{w_i}\cdot u+\Delta_{i+1}\cdot u+\sum_{k=i+1}^{j-1}\Delta_{k+1}\cdot u
                \leq A+\sum_{k=i+1}^{j-1}D_{k+1}^{u, A}
        \end{equation*}
    This implies that
        \begin{equation}
            \label{eq_heoffding_setup}
            \P[G_A,~B_u^{i, j},~\pi_{w_j}\cdot u>t]\leq\P\left[\sum_{k=i+1}^{j-1}D_{k+1}^{u, A}>t-A\right].
        \end{equation}
    By construction, for each $k\geq 0$, the random variable $D_{k+1}^{u, A}$ is $\G_{k+1}$-measurable, bounded by $A$ and using Proposition \ref{prop_symmetrization_decomposition}, 
        \[\E[D_{k+1}^{u, A}~|~\G_k]=0.\]
    Therefore, as $j-i-1\leq \|\pi_0\|$, one can apply Hoeffding's concentration inequality in \eqref{eq_heoffding_setup} to get that
        \[\P[G_A,~B_u^{i, j},~\pi_{w_j}\cdot u>t]\leq\exp\left[-\frac{(t-A)^2}{2\|\pi_0\|A^2}\right].\]
    Injecting in \eqref{eq_proba_decomposition} it comes
        \begin{equation}
            \label{eq_p_ga_sup}
            \P\left[G_A,~\sup_{n\geq0}\|\mathbf p_{\perp \pi_0}(\pi_{w_n})\|>t\sqrt{2(d-1)}\right]\leq (d-1)\|\pi_0\|(\|\pi_0\|+1)\exp\left[-\frac{(t-A)^2}{2\|\pi_0\|A^2}\right].
        \end{equation}
    Then, using a union bound, one obtains
        \begin{equation}
            \label{eq_before_theta}
            \begin{split}
                &\P\left[\sup_{0\leq n\leq \Theta}\|\mathbf p_{\perp\pi_0}(\pi_n)\|>t\sqrt{2(d-1)}+A\right]\\
                &\qquad\leq\P\left[\sup_{n\geq0}\sum_{k=w_n}^{w_{n+1}-1}\|\pi_{k+1}-\pi_k\|>A\right]+\P\left[G_A,~\sup_{n\geq0}\|\mathbf p_{\perp \pi_0}(\pi_{w_n})\|>t\sqrt{2(d-1)}\right]\\
                &\qquad\leq C_\Sigma\|\pi_0\|\exp\left[-c_\Sigma\sqrt A\right]+(d-1)\|\pi_0\|(\|\pi_0\|+1)\exp\left[-\frac{(t-A)^2}{2\|\pi_0\|A^2}\right],
            \end{split}
        \end{equation}
    Where the last line follows from Proposition \ref{prop_fluctuation_control} and \eqref{eq_p_ga_sup}. Finally, observing that 
        \[\sup_{n\geq\Theta}\|\mathbf p_{\perp \pi_0}(\pi_n)\|\leq\sup_{n\geq\Theta}\|\pi_n\|=R_\Theta,\]
    one derives
        \begin{equation}
            \label{eq_dev_final}
            \begin{split}
                &\P\left[\sup_{n\geq0}\|\mathbf p_{\perp \pi_0}(\pi_n)\|>t\sqrt{2(d-1)}+A\right]\\
                &\qquad\leq\P[R_\Theta>t]+\P\left[\sup_{0\leq n\leq \Theta}\|\mathbf p_{\perp\pi_0}(\pi_n)\|>t\sqrt{2(d-1)}+A\right]\\
                &\qquad\leq C_\Theta\exp\left[-c_\Theta t^{\frac{d}{d+1}}\right]+C_\Sigma\|\pi_0\|\exp\left[-c_\Sigma\sqrt A\right]+(d-1)\|\pi_0\|(\|\pi_0\|+1)\exp\left[-\frac{(t-A)^2}{2\|\pi_0\|A^2}\right],
            \end{split}
        \end{equation}
    where the last inequality follows from \eqref{eq_before_theta} and Proposition \ref{prop_r_theta}. Finally, considering $\|\pi_0\|\geq 1$, fixing $\varepsilon\in(0,\frac{1}{2})$ and setting 
        \[t\coloneqq\|\pi_0\|^{\frac{1}{2}+\varepsilon}\quad\text{as well as}\quad A\coloneqq t^\delta
        \quad\text{with}\quad\delta\coloneqq\frac{4}{5}\frac{2\varepsilon}{1+2\varepsilon},\]
    one obtains
        \[t^{\frac{d}{d+1}}=\|\pi_0\|^{\frac{d(1+2\varepsilon)}{2(d+1)}},\quad t\sqrt{2(d-1)}+A=\mathcal O_{\pi_0\to\infty}\left(\|\pi_0\|^{\frac{1}{2}+\varepsilon}\right),\]
    and
        \[\sqrt A\sim_{\pi_0\to\infty}\|\pi_0\|^{\frac{\delta}{4}(1+2\varepsilon)}=\|\pi_0\|^{\frac{2}{5}\varepsilon}=\|\pi_0\|^{2\varepsilon-\delta(1+2\varepsilon)}\sim_{\pi_0\to\infty}\frac{(t-A)^2}{\|\pi_0\|A^2},\]
    which injected in \eqref{eq_dev_final} implies that one can find constants $c_\varepsilon, C_\varepsilon>0$ such that for all $\pi_0\in\R^d$,
        \[\P\left[\sup_{n\geq0}\|\mathbf p_{\perp \pi_0}(\pi_n)\|>\|\pi_0\|^{\frac{1}{2}+\varepsilon}\right]\leq C_\varepsilon\exp\left[-c_\varepsilon\|\pi_0\|^{\frac{2}{5}\varepsilon}\right],\]
    where one used that $\frac{d(1+2\varepsilon)}{2(d+1)}\geq\frac{1}{2}+\varepsilon\geq\frac{2}{5}\varepsilon$ as $d\geq1$. This concludes the proof.
\end{proof}

\bibliographystyle{amsalpha}
\bibliography{bibliography}

\end{document}